\documentclass[a4paper,leqno,11pt]{amsart}

\usepackage{amsmath}
\usepackage{amscd}
\usepackage{amssymb}
\usepackage{enumerate}
\usepackage{indentfirst}
\usepackage{latexsym}
\usepackage{multicol}
\usepackage{amsthm}
\usepackage{color}

\newcommand{\Sum}{\displaystyle \sum} 
 
\newcommand{\Int}{\displaystyle \int} 
\newcommand{\Frac}{\displaystyle \frac}

\numberwithin{equation}{section}
\def\div{ \hbox{\rm div}\,  }
\def\curl{ \hbox{\rm curl}\,  }

\def\Id{\hbox{\rm Id}}
\def\adj{\hbox{\rm adj}}
\def\Tr{\hbox{\rm Tr}}
\def\divA{\, \hbox{\rm div}_A\,  }
\def\divx{\, \hbox{\rm div}_x\,  }
\def\divy{\, \hbox{\rm div}_y\,  }
\newcommand{\R}{{\mathbb R}}
\newcommand{\N}{{\mathbb N}}

\newcommand{\Z}{{\mathbb Z}}

\def\d{\partial}

\def\ep{\varepsilon}

\def\ddj{\dot\Delta_j}

\def\cP{{\mathcal P}}

\def\cE{{\mathcal E}}

\def\cM{{\mathcal M}}
\def\cC{{\mathcal C}}
\def\cS{{\mathcal S}}
\let\tilde=\widetilde

\newcommand{\dC}{\delta\! C}
\newcommand{\da}{\delta\! a}
\newcommand{\df}{\delta\! f}

\newcommand{\du}{\delta\! u}
\newcommand{\dv}{\delta\! v}

\newcommand{\dr}{\delta\!\rho}

\newtheorem{lem}{Lemma}

\newtheorem{prop}{Proposition}
\newtheorem{theo}{Theorem}
\newtheorem{rem}{Remark}

\newenvironment{p}{
\noindent\textit{\textbf{Proof:}}~}
{\hfill\rule{2.1mm}{2.1mm}
}

\begin{document}
\title[Compressible Navier-Stokes equations]{A Lagrangian approach for the compressible Navier-Stokes equations\\[3ex]
\emph{\small Une approche lagrangienne pour le syst\`eme de Navier-Stokes compressible}}

\author[R. Danchin]{Rapha\"el Danchin}
\address[R. Danchin]
{Universit\'e Paris-Est, LAMA, UMR 8050 and Institut Universitaire de France,
 61 avenue du G\'en\'eral de Gaulle,
94010 Cr\'eteil Cedex, France.}
\email{danchin@univ-paris12.fr}

\date\today  
\begin{abstract} Here we  investigate the Cauchy problem for  the barotropic Navier-Stokes equations in $\R^n$,  in the critical Besov spaces setting.
We improve recent results (see \cite{CMZ1,D1,D2})  as regards the uniqueness condition:  initial velocities in critical  Besov spaces with (not too) \emph{negative} indices  
 generate a unique local solution. 
Apart from (critical) regularity, the initial density just has to be bounded away from $0$ and to tend to some
positive constant at infinity. Density-dependent
viscosity coefficients may be considered. 
Using Lagrangian coordinates is the key to our statements as it enables us to solve the
system by means of the basic contraction mapping theorem. As a consequence, 
conditions for uniqueness are the same as for existence, and Lipschitz continuity of the flow map 
(in Lagrangian coordinates) is established.
\bigbreak\noindent R\'ESUM\'E. On \'etudie le probl\`eme de Cauchy pour le syst\`eme de Navier-Stokes barotrope
dans $\R^n,$ avec r\'egularit\'e Besov critique.  On affaiblit la condition d'unicit\'e  des articles  \cite{CMZ1,D1,D2}, ce qui permet d'\'etablir
entre autres que des vitesses initiales ayant une r\'egularit\'e Besov (pas trop) n\'egative g\'en\`erent une solution unique.  La densit\'e initiale
est \`a r\'egularit\'e critique et doit juste \^etre strictement positive et tendre vers une constante \`a l'infini. Les coefficients de viscosit\'e 
preuvent d\'ependre de la densit\'e. L'usage de coordonn\'ees lagrangiennes est la clef de toutes ces am\'eliorations
car il permet de r\'esoudre le syst\`eme par it\'erations de Picard. Comme corollaire imm\'ediat, on obtient que les conditions
pour l'unicit\'e sont les m\^emes que pour l'existence, ainsi que la continuit\'e de l'op\'erateur solution (pour le syst\`eme \'ecrit
en coordonn\'ees lagrangiennes). 
\end{abstract}


\maketitle

\section*{Introduction}

 We address the well-posedness issue for the barotropic  compressible Navier-Stokes equations 
with variable density in the whole space $\R^n$: 
\begin{equation}\label{eq:euler}
\left\{\begin{array}{l}
\d_t\rho+\div(\rho u)=0,\\[1ex]
\d_t(\rho u)+\div(\rho u\otimes  u)-2\div(\mu(\rho)D(u))-\nabla(\lambda(\rho)\div u)+\nabla(P(\rho))=0,\\[1ex]
\rho|_{|t=0}=\rho_0,\qquad u|_{|t=0}=u_0.
\end{array}\right.
\end{equation}
Above $\rho=\rho(t,x)\in\R_+$ stands for the density, 
$u=u(t,x)\in\R^n,$ for the velocity field. The space variable $x$ belongs to the whole
$\R^n.$   The notation $D(u)$ designates the \emph{deformation tensor}
which is defined by $$D(u):=\frac12(Du+\nabla u)\quad\hbox{with}\quad
(Du)_{ij}:=\d_ju^i\quad\hbox{and}\quad (\nabla u)_{ij}:=\d_iu^j.$$ 
The pressure function $P$ and the viscosity coefficients $\lambda$ and $\mu$ are given suitably smooth functions of the density. With no loss of generality, one may assume that $P$ is defined over $\R$ and vanishes at $0.$
 As we focus on \emph{viscous} fluids, we suppose  that
\begin{equation}\label{eq:lamepos}
\alpha:=\min\Bigl(\inf_{\rho>0}  (\lambda(\rho)+2\mu(\rho)),\inf_{\rho>0} \mu(\rho)\Bigr)>0,
\end{equation}
which ensures   the second order operator in the velocity equation of \eqref{eq:euler} to be uniformly elliptic. 
\medbreak
We supplement System \eqref{eq:euler} with the condition at infinity that
 $u$   tends to $0$ and $\rho,$ to some positive constant (that may be taken equal 
 to $1$ after suitable normalization).  The exact meaning of those boundary conditions
  will be given by the functional framework in which we shall consider the system. 
\smallbreak

 In the present paper, we aim at solving \eqref{eq:euler}
 in \emph{critical} functional spaces, that is  in spaces which have 
 the same invariance with respect to time and space dilation as 
 the system itself (see e.g. \cite{D1}  for more explanations
 about this nowadays classical approach). In this framework, it has been 
 stated \cite{D1,D2} in the constant coefficients case  
 that,  for data $(\rho_0,u_0)$ such that 
$$a_0:=(\rho_0-1)\in\dot B^{n/p}_{p,1}(\R^n),\qquad
u_0\in\dot B^{n/p-1}_{p,1}(\R^n)$$ 
and that, for a small enough constant $c,$
\begin{equation}\label{eq:smalldata0}
\|a_0\|_{\dot B^{n/p}_{p,1}(\R^n)}\leq c,
\end{equation}
we have for any  $p\in[1,2n)$: 
\begin{itemize}
\item existence of a local solution $(\rho,u)$ 
such that   $a:=(\rho-1)\in\cC_b([0,T];\dot B^{n/p}_{p,1}),$
$u\!\in\!\cC_b([0,T];\dot B^{n/p-1}_{p,1})$ and $\d_tu,\nabla^2u\in L^1(0,T;\dot B^{n/p-1}_{p,1})$;
\item uniqueness in the above space if in addition  $p\leq n.$
\end{itemize}
If $p\leq n$ then the viscosity coefficients may depend (smoothly) on $\rho$ and
the smallness condition \eqref{eq:smalldata0} may be replaced by the following positivity condition (see \cite{CMZ1,D3}):
\begin{equation}
\inf_{x\in\R^n} \rho_0(x)>0.
\end{equation}
Those results have been somewhat extended in \cite{H2} where it has been noticed 
that $a_0$ may be taken in a larger Besov space, with another Lebesgue exponent.
\medbreak
The above  results are based on maximal regularity estimates in Besov spaces for the evolutionary Lam\'e system, 
and on the Schauder-Tychonoff fixed point theorem.
In effect, owing to the hyperbolicity of the density equation, 
there is a loss of one derivative in the stability estimates
thus precluding the use of the contraction mapping (or Banach fixed point) theorem. 
As a consequence, with  this method it is  found that the conditions for uniqueness are \emph{stronger} than those
for  existence. 
\medbreak
Following our recent paper \cite{DM-cpam} dedicated to the incompressible density-dependent Navier-Stokes equation, and older works concerning the compressible Navier-Stokes equations
(see  \cite{Mucha,valli,VZ}), 
 we here aim at solving System \eqref{eq:euler} in the \emph{Lagrangian coordinates.} 
The main motivation is that the mass  is \emph{constant} along the flow hence, to some extent,
only the (parabolic type) equation for the velocity has to be considered. 
After performing this change of coordinates, we shall see that  solving
\eqref{eq:euler} may be done by means of the Banach fixed point theorem. 
Hence, the condition for uniqueness \emph{is the same} as that for the existence, and  the flow map is Lipschitz continuous.
In addition, in the case of fully nonhomogeneous fluids with variable viscosity coefficients, 
the analysis turns out to be simpler than in \cite{CMZ1,D3} 
even for density-dependent viscosity coefficients and in the case 
where the density is not close to a constant. Indeed, our  proof
relies essentially on a priori estimates for a parabolic system (a suitable linearization of 
the momentum equation in Lagrangian coordinates) with rough constant \emph{depending only
on the initial density} hence time-independent. In contrast, in \cite{CMZ1,D3}
 tracking  the time-dependency of the coefficients was quite  technical. 
\medbreak
We now come to the plan  of the paper. 
In the next section, we introduce the compressible Navier-Stokes equations in Lagrangian coordinates
and present our main results.
Section \ref{s:simple} is devoted to the proof of our main existence and uniqueness result 
in the simpler case where the density is close to a constant and the coefficients, density independent.
In Section \ref{s:nonhomo}, we treat the general fully nonhomogeneous case with 
nonconstant coefficients. A great deal of the analysis is contained in the study of  the linearized momentum equation for  \eqref{eq:euler}
(see Subsection \ref{s:linear}) which turns out to be a Lam\'e type system with variable rough coefficients.
This will enable us to define a self-map $\Phi$ on a suitably small ball of 
some  Banach space  $E_p(T)$ and to apply  the contraction mapping
theorem so as to solve the compressible Navier-Stokes equations in Lagrangian coordinates.
%Some related problems and open questions (in particular that of the global well-posedness issue in this framework) are postponed in Section \ref{s:further}. 
  In the Appendix we prove several technical  results concerning
the Lagrangian coordinates and Besov spaces.

\medbreak
\noindent {\bf Notation:}  
Throughout, the notation $C$ stands for a generic constant (the meaning of which depends on the context), and we sometimes write $A\lesssim B$ instead of $A\leq CB.$ For $X$ a Banach space,
$p\in[1,+\infty]$ and $T>0,$ the notation  $L^p(0,T;X)$ or $L^p_T(X)$  designates  
  the set of measurable functions $f:[0,T]\rightarrow X$ with $t\mapsto\|f(t)\|_X$ in $L^p(0,T),$
  endowed with the norm
  $$
  \|f\|_{L^p_T(X)}:=\bigl\|\,\| f\|_{X}\,\bigr\|_{L^p(0,T)}.
  $$
  We agree that  $\cC([0,T];X)$ denotes the set of continuous functions from $[0,T]$ to $X.$

%%%%%%%%%%%%%%%%%%%%%%%%%%%%%%%%%%%%%

 \section{Main results}

 Before  deriving   the Lagrangian equations
corresponding to \eqref{eq:euler}, let us introduce more notation.
We agree that for a $C^1$ function $F:\R^n\rightarrow \R^n\times\R^m$ then 
$\div F:\R^n\rightarrow\R^m$ with
$$
(\div F)^j:=\sum_i\partial_iF_{ij}\quad\hbox{for }\ 1\leq j\leq m,
$$ 
and that  for $A=(A_{ij})_{1\leq i,j\leq n}$ and  $B=(B_{ij})_{1\leq i,j\leq n}$
 two $n\times n$ matrices, we denote $$A:B=\Tr AB=\sum_{i,j} A_{ij} B_{ji}.$$ 
 The notation $\adj(A)$ designates the
 adjugate matrix that is the transposed cofactor matrix. 
 Of course if $A$ is invertible then we have
 $\adj(A)=(\det A)\:A^{-1}.$
 Finally, given some matrix $A,$ we define the ``twisted'' deformation tensor
 and divergence operator (acting on vector fields $z$) by the formulae
 $$
 D_A(z):=\frac12\bigl(Dz\cdot A+{}^T\!A\cdot\nabla z\bigr)\quad\hbox{and}\quad
 \divA z:={}^T\!A:\nabla z=Dz:A.
  $$
Let $X$ be the flow associated to the vector-field $u,$ that is
the solution to 
\begin{equation}\label{lag}
X(t,y)=y+\int_0^tu(\tau,X(\tau,y))\,d\tau.
\end{equation}
Denoting 
$$
\bar\rho(t,y):=\rho(t,X(t,y))\ \hbox{ and }\ 
\bar u(t,y)=u(t,X(t,y))
$$
with $(\rho,u)$ a solution of \eqref{eq:euler}, 
and using the chain rule and Lemma \ref{l:div} from the Appendix, we gather
that $(\bar\rho,\bar u)$  satisfies
 \begin{equation}\label{eq:lagrangian}
\left\{\begin{array}{l}
\d_t(J\bar\rho)=0\\[1.5ex]
\rho_0\d_t\bar u-\div\Bigl(\adj(DX)\bigl(2\mu(\bar\rho)D_A(\bar u)
+\lambda(\bar\rho)\divA\bar u\,\Id+P(\bar\rho)\Id\bigr)\Bigr)=0\end{array}\right.
\end{equation}
with  $J:=\det DX$ and $A:=(D_yX)^{-1}.$ 
Note that one may forget any reference to the initial Eulerian vector-field $u$ by defining directly the ``flow'' 
$X$  of  $\bar u$ by the formula
\begin{equation}\label{eq:lag}
X(t,y)=y+\int_0^t\bar u(\tau,y)\,d\tau.
\end{equation}

We want to solve the above system in \emph{critical} homogeneous Besov spaces.
Let us recall that, for $1\leq p\leq\infty$ and $s\leq n/p,$ a tempered distribution $u$ 
over $\R^n$ 
belongs to the homogeneous Besov space $\dot B^s_{p,1}(\R^n)$ if 
$$
u=\sum_{j\in\Z}\ddj u\quad\hbox{in }\ \cS'(\R^n)
$$
and 
\begin{equation}\label{eq:B}
\|u\|_{\dot B^s_{p,1}(\R^n)}:=\sum_{j\in\Z}2^{js}\|\ddj u\|_{L^p(\R^n)}<\infty.
\end{equation}
Here $(\ddj)_{j\in\Z}$ denotes a homogeneous dyadic resolution of unity 
in Fourier variables --the so-called Littlewood-Paley decomposition (see e.g. \cite{BCD}, Chap. 2 for more details on the
Littlewood-Paley decomposition and Besov spaces). 
\smallbreak
Loosely speaking, a function belongs to $\dot B^s_{p,1}(\R^n)$
if it has $s$ derivatives in $L^p(\R^n).$ 
In the present paper, we shall mainly use  the following classical properties:
\begin{itemize}
\item the Besov space $\dot B^{n/p}_{p,1}(\R^n)$ is a Banach algebra embedded in 
the set of continuous functions going to $0$ at infinity, whenever $1\leq p<\infty;$
\item the usual product maps $\dot B^{n/p-1}_{p,1}(\R^n)\times \dot B^{n/p}_{p,1}(\R^n)$
in $\dot B^{n/p-1}_{p,1}(\R^n)$ whenever $1\leq p<2n;$
\item Let $F:I\rightarrow\R$ be a smooth function (with $I$ an open  interval of $\R$ containing $0$)
vanishing at $0.$ Then for any $s>0,$  $1\leq p\leq\infty$ and interval $J$ compactly supported in $I$ there exists
a constant $C$ such that 
\begin{equation}\label{eq:compo1}
\|F(a)\|_{\dot B^{s}_{p,1}(\R^n)}\leq C\|a\|_{\dot B^{s}_{p,1}(\R^n)}
\end{equation}
for any $a\in\dot B^s_{p,1}(\R^n)$ with values in $J.$
 In addition, if $a_1$ and $a_2$ are two such functions and $s=n/p$ then we have
 \begin{equation}\label{eq:compo2}
\|F(a_2)-F(a_1)\|_{\dot B^{n/p}_{p,1}(\R^n)}\leq C\|a_2-a_1\|_{\dot B^{n/p}_{p,1}(\R^n)}.
\end{equation}
\end{itemize}

{}From now on, we shall omit $\R^n$ in the notation for Besov spaces. 
We shall obtain the existence and uniqueness of a local-in-time  solution $(\bar\rho,\bar u)$ 
for \eqref{eq:lagrangian}, with $\bar a:=\bar\rho-1$ in $\cC([0,T];\dot B^{n/p}_{p,1})$ and $\bar u$ in the space
$$
E_p(T):=\bigl\{v\in\cC([0,T];\dot B^{n/p-1}_{p,1}), 
\ \d_tv,\nabla ^2v\in L^1(0,T;\dot B^{n/p-1}_{p,1})\bigr\}\cdotp
$$
That space will be  endowed  with the norm
$$
\|v\|_{E_p(T)}:=\|v\|_{L^\infty_T(\dot B^{n/p-1}_{p,1})}+
\|\d_tv,\nabla^2v\|_{L^1_T(\dot B^{n/p-1}_{p,1})}.
$$
 Let us now state our main result.
\begin{theo}\label{th:main1}
 Let $1<p<2n$ and $n\geq2.$ Let  $u_0$ be a vector-field in $\dot B^{n/p-1}_{p,1}.$
 Assume that the initial density $\rho_0$ satisfies $a_0:=(\rho_0-1)\in\dot B^{n/p}_{p,1}$ and   
  \begin{equation}\label{eq:positive}
 \inf_{x}\rho_0(x)>0.
 \end{equation}
 Then System \eqref{eq:lagrangian} has a unique local  solution $(\bar\rho,\bar u)$ with  
$(\bar a,\bar u)\in\cC([0,T];\dot B^{n/p}_{p,1})\times E_p(T).$
Moreover, 
 the flow map $(a_0,u_0)\longmapsto (\bar a,\bar u)$
is Lipschitz  continuous from  $\dot B^{n/p}_{p,1}\times\dot B^{n/p-1}_{p,1}$
to  $\cC([0,T];\dot B^{n/p}_{p,1})\times E_p(T).$
\medbreak
If $\rho_0$ is close enough to some positive constant then the statement holds true
for all $p\in[1,\infty)$ and $n\geq1.$
 \end{theo} 
 In Eulerian coordinates, this result recasts in:
\begin{theo}\label{th:main2}
Under the hypotheses of Theorem \ref{th:main1} with $1<p<2n$ and $n\geq2,$  System \eqref{eq:euler} has a unique local solution $(\rho, u)$ 
with~ $u\in E_p(T),$ $\rho$ bounded away from $0$ and  $(\rho-1)\in \cC([0,T];\dot B^{n/p}_{p,1}).$ 
 \end{theo}
 Let us make a few comments concerning the above assumptions.
\begin{itemize}
\item We expect the Lagrangian method  to improve the uniqueness conditions given in e.g. \cite{D1} for
the full Navier-Stokes equations. We here consider the barotropic case for simplicity.
\item
The condition $1\leq p<2n$ is a consequence of the product laws in Besov spaces. 
It implies  that the regularity exponent for the velocity has to be greater than $-1/2$
(to be compared with $-1$ for the homogeneous incompressible Navier-Stokes equations).
It would be interesting to see whether introducing a modified velocity as in B. Haspot's works \cite{H,H2} 
allows to consider \emph{different} Lebesgue exponents for the Besov spaces pertaining to the density 
and the velocity so as to go beyond $p=2n$ for the velocity.

\item The regularity condition over the density is \emph{stronger} than that
for density-dependent incompressible fluids (see \cite{DM-cpam}). In particular, in contrast with incompressible fluids, 
it is not clear that combining  Lagrangian coordinates and critical regularity approach allows 
 to consider \emph{discontinuous} densities.  
\item Owing to the fact that the density satisfies a transport equation, 
we do not expect Lipschitz continuity of the flow map in high norm for the Eulerian formulation
to be true. 

\item It is  worth comparing our results with those of P. Germain in \cite{Germain}, 
 and D. Hoff in \cite{Hoff} concerning the weak-strong uniqueness problem.  
In both papers, the idea is to show that, \emph{in the constant viscosity case},  a finite energy weak solution  coincides 
with a strong one under some additional assumptions. The weak solution turns out to 
have less regularity than in Theorem \ref{th:main2}. At the same time, the assumptions 
on the strong solution $(\rho,u)$  are much stronger. In both papers, $\nabla u$ has to be in $L^1(0,T;L^\infty),$
and to satisfy additional conditions: roughly  $\nabla^2 u$ or $\d_tu$ have to be in $L^2(0,T;L^{d})$ 
in Germain's work, while  
$\sqrt t D^2u\in L^r(0,T;L^4)$ with $r=4/3$ if $n=2,$ and $r=8/5$ if $n=3$ in Hoff's paper. 
Some regularity conditions are  required on the density but they are, to some extent, weaker
than ours. 
\end{itemize}

%%%%%%%%%%%%%%%%%%%%%%%%%%%%%%%%%%%%%%%%%%

\section{The simple case of almost homogeneous compressible fluids} \label{s:simple}

As a warm up and for the reader convenience, we here explain how local well-posedness may be proved  for 
the system in Lagrangian coordinates in the simple case where: 
\begin{enumerate}
\item The viscosity coefficients are constant,
\item The density is very close to one.
\end{enumerate}
 Let $\mu':=\lambda+\mu.$ 
Keeping in mind the above  two conditions and using the fact that 
the first equation of \eqref{eq:lagrangian}  implies that 
\begin{equation}\label{eq:rhobar} J(t,\cdot)\bar\rho(t,\cdot)\equiv\rho_0,
\end{equation}
with  $J:=|\det DX|$   and 
\begin{equation}\label{eq:flowbar}
X(t,y):=y+\int_0^t\bar u(\tau,y)\,d\tau,
\end{equation}
 we rewrite the equation  for  the Lagrangian velocity as (recall that  $A:=(DX)^{-1}$):
\begin{multline}\label{eq:2.3}
\d_t\bar u-\mu\Delta\bar u-\mu'\nabla\div\bar u=(1-\rho_0)\d_t\bar u
+2\mu\,\div\bigl(\adj(DX)D_A(\bar u)-D(\bar u)\bigr)\\
+\lambda\div\bigl(\adj(DX)\divA\bar u-\div\bar u\:\Id\bigr)-\div\bigl(\adj(DX)P(J^{-1}\rho_0)\bigr).
\end{multline}

The left-hand side of the above equation is the linear Lam\'e system with constant coefficients,
the solvability of which may be easily deduced from that of the heat equation  in the whole space (see e.g. \cite{BCD}, Chap. 2
or \cite{D-Chambery}).
We get:
\begin{prop}\label{p:lamewhole}
Let the viscosity coefficients $(\mu,\mu')\in\R^2$ satisfy $\mu>0$
and $\mu+\mu'>0.$
 Let $p\in [1,\infty]$ and $ s\in\R.$ Let  $u_0\in \dot B^s_{p,1}$
 and  $f \in L^1(0,T;\dot B^s_{p,1}).$ 
 Then the Lam\'e system
 \begin{equation}\label{eq:lame}
\left\{ \begin{array}{lcr}
 \d_tu-\mu \Delta u -\mu'\nabla\div u=f \qquad & \mbox{in} & (0,T)\times\R^n\\[5pt]
u|_{t=t_0} = u_0 &  \mbox{on} & \R^n
\end{array}\right.
\end{equation}
 has a unique solution $u$ in $\cC([0,T);\dot B^s_{p,1})$ such that
 $ \d_tu,\nabla^2u\in L^1(0,T;\dot B^s_{p,1})$
and the following estimate is valid:
\begin{equation}\label{eq:lameest}
\|u\|_{L^\infty_T(\dot B^s_{p,1})}+\min(\mu,\mu+\mu')\|\nabla^2 u\|_{ L^1_T(\dot B^s_{p,1})}
\leq C(\|f\|_{ L^1_T(\dot B^s_{p,1})}+\|u_0\|_{\dot B^s_{p,1}})
\end{equation}
where $C$ is an absolute constant with no dependence on $\mu,\mu'$ and $T.$
\end{prop}
In the rest of this section, we drop the bars on the Lagrangian velocity field. 
Granted with the above proposition, we define a map $\Phi: v\mapsto u$ on $E_p(T)$  where 
$u$ stands for the solution to 
\begin{equation}\label{eq:baru}
\d_tu-\mu\Delta u-\mu'\nabla\div u= I_1(v)+2\mu\div I_2(v,v)
+\lambda\div I_3(v,v)-\div I_4(v)
\end{equation}
with initial data $u_0$ and 
$$
\begin{array}{llllll}
I_1(w)&=&-a_0\d_tw,\quad&
I_2(v, w)&=&\adj(DX_v)D_{A_v}(w)-D(w),\\[1ex]
I_3(v, w)&=&\div_{\!A_v}w\:\adj(DX_v)-\div w\:\Id,\quad&
I_4(v)&=&\adj(DX_v)P(J_v^{-1}\rho_0).
\end{array}
$$
Note that any fixed point of $\Phi$ is a solution in $E_p(T)$ to \eqref{eq:2.3}. 
We claim that the existence of such points is a consequence of 
 the standard Banach fixed point theorem in a suitable closed ball of $E_p(T).$ 
\smallbreak
\subsubsection*{First step: estimates for $I_1,$ $I_2,$ $I_3$ and $I_4$}

Throughout we assume that  for a small enough constant $c,$ 
\begin{equation}\label{eq:smallv}
\int_0^T\|Dv\|_{\dot B^{n/p}_{p,1}}\,dt\leq c.
\end{equation}
It is obvious that
\begin{equation}\label{eq:I1}
\|I_1(w)\|_{L^1_T(\dot B^{n/p-1}_{p,1})}\leq \|a_0\|_{\cM(\dot B^{n/p-1}_{p,1})}
\|\d_tw\|_{L^1_T(\dot B^{n/p-1}_{p,1})}
\end{equation}
where  the multiplier norm  $\cM(\dot B^{s}_{p,1})$ for $\dot B^{s}_{p,1},$
 is defined by 
\begin{equation}\label{eq:defmult}
\|f\|_{\cM(\dot B^{s}_{p,1})}:=\sup \|\psi f\|_{\dot B^{s}_{p,1}}.
\end{equation}
 The supremum is taken over those functions $\psi$ in $\dot B^{s}_{p,1}$ with norm $1.$
\medbreak
Next, 
taking advantage of the fact that $\dot B^{n/p}_{p,1}$ is an algebra if $1\leq p<\infty$, 
of  \eqref{eq:U3}, \eqref{eq:U4}
and \eqref{eq:smallv}, we readily get
 \begin{equation}\label{eq:I2}
\|I_2( v, w)\|_{L^1_T(\dot B^{n/p}_{p,1})}+\|I_3(v,w)\|_{L^1_T(\dot B^{n/p}_{p,1})}\leq 
C\|Dv\|_{L^1_T(\dot B^{n/p}_{p,1})}\|D w\|_{L^1_T(\dot B^{n/p}_{p,1})}.
\end{equation}
As regards the pressure term (that is $I_4(v)$), we use the fact that
 under assumption \eqref{eq:smallv}, we have, by virtue of the composition inequality \eqref{eq:compo1} and of
   flow estimates (see  \eqref{eq:U1} and \eqref{eq:J}),
  \begin{equation}\label{eq:I4}
  \|I_4(v)\|_{L_T^\infty(\dot B^{n/p}_{p,1})}
  \leq C\bigl(1+\|D v\|_{L_T^1(\dot B^{n/p}_{p,1})}\bigr)\bigl(1+\|a_0\|_{\dot B^{n/p}_{p,1}}\bigr).
  \end{equation}

\subsubsection*{Second step: $\Phi$ maps a suitable closed ball in itself}
At this stage, one may assert that if $v\in E_p(T)$ satisfies \eqref{eq:smallv} 
then the right-hand side of \eqref{eq:baru} belongs to $L^1(0,T;\dot B^{n/p-1}_{p,1}).$ Hence
Proposition \ref{p:lamewhole} implies that $\Phi(v)$ is well defined and maps $E_p(T)$ to itself. 
However it is not clear that it is contractive over the whole set $E_p(T).$ 
So we introduce  $u_L$  the ``free solution'' to 
$$
\d_tu_L-\mu\Delta u_L-\mu'\nabla\div u_L=0,\qquad u_L|_{t=0}=u_0.
$$
Proposition \ref{p:lamewhole} guarantees that $u_L$ belongs to $E_p(T)$
for all $T>0.$ 
\medbreak
 We  claim that  if $T$ is small enough (a condition which will be expressed in terms of
  the free solution $u_L$) and if $R$
is small enough  (a condition which will depend only on the viscosity coefficients
  and on $p,$ $n$ and $P$) then
  $v\in \bar B_{E_p(T)}(u_L,R)$ implies that \eqref{eq:smallv} is fulfilled and that  $u\in \bar B_{E_p(T)}(u_L,R).$
  Indeed     $\tilde u:= u-u_L$  satisfies
   $$
   \left\{\begin{array}{l}
   \d_t\tilde u-\mu\Delta\tilde u-\mu'\nabla\div\tilde u=
 I_1(v)+2\mu\div I_2(v, v)
+\lambda\div I_3(v,v)-\div I_4(v),\\[1ex]
\tilde  u|_{t=0}=0.\end{array}\right.
$$
So Proposition  \ref{p:lamewhole} yields\footnote{For simplicity, we do not track the dependency
of the coefficients with respect to $\mu$ and $\mu'.$} 
   $$
  \|\tilde u\|_{E_p(T)}\lesssim \|I_1(v)\|_{L_T^1(\dot B^{n/p-1}_{p,1})}
  + \|I_2( v,v)\|_{L_T^1(\dot B^{n/p}_{p,1})}  +
 \|I_3(v, v)\|_{L_T^1(\dot B^{n/p}_{p,1})}
  +T\|I_4(v)\|_{L_T^\infty(\dot B^{n/p}_{p,1})}.
  $$
Inserting inequalities \eqref{eq:I1}, \eqref{eq:I2}  and \eqref{eq:I4}, we thus get:
  $$
  \displaylines{  \|\tilde u\|_{E_p(T)}\lesssim 
\|Dv\|_{L_T^1(\dot B^{n/p}_{p,1})}^2
 +\|a_0\|_{\cM(\dot B^{n/p-1}_{p,1})}
\|\d_tv\|_{L_T^1(\dot B^{n/p-1}_{p,1})} +T(1+\|a_0\|_{\dot B^{n/p}_{p,1}}).}  
  $$
That is, keeping in mind that $v$ is in $\bar B_{E_p(T)}(u_L,R),$
  $$\displaylines{
  \|\tilde u\|_{E_p(T)}\leq C\Bigl(\|a_0\|_{\cM(\dot B^{n/p-1}_{p,1})}
  (R+ \|\d_tu_L\|_{L_T^1(\dot B^{n/p-1}_{p,1})})%\hfill\cr\hfill
  +\|Du_L\|_{L_T^1(\dot B^{n/p}_{p,1})}^2+R^2
  +T(1+\|a_0\|_{\dot B^{n/p}_{p,1}})\Bigr).}  
  $$
  So we see that  if   $T$ satisfies
    \begin{equation}\label{eq:smallT}
  CT(1+\|a_0\|_{\dot B^{n/p}_{p,1}})\leq R/2\quad\hbox{and}\quad
  \|Du_L\|_{L_T^1(\dot B^{n/p}_{p,1})}+ \|\d_tu_L\|_{L_T^1(\dot B^{n/p-1}_{p,1})}\leq  R
 \end{equation}
 then we have 
 $$
   \|\tilde u\|_{E_p(T)}\leq 2C\|a_0\|_{\cM(\dot B^{n/p-1}_{p,1})}R+2CR^2+R/2.
   $$
  Hence     there exists a small constant $\eta=\eta(n,p)$ such that if 
    \begin{equation}\label{eq:smallness1}
  \|a_0\|_{\cM(\dot B^{n/p-1}_{p,1})}\leq\eta,
  \end{equation}
  and if $R$ has been chosen small enough then $u$ is in  $\bar B_{E_p(T)}(u_L,R).$
  Of course, taking $R$ and $T$ even smaller ensures that \eqref{eq:smallv}
  is satisfied for all vector-field of  $\bar B_{E_p(T)}(u_L,R).$
  
  %%%%%%%%%%%%%%%%%%%%%%%%%%%%%%%%%%%  
  
\subsubsection*{Third step: contraction properties} 

We  claim that under Conditions \eqref{eq:smallness1} and  \eqref{eq:smallT}
(with a smaller  $R$ if needed), 
  the map $\Phi$ is  $1/2$-Lipschitz   over $\bar B_{E_p(T)}(u_L,R).$  
So we are given  $v_1$ and $v_2$ 
in $\bar B_{E_p(T)}(u_L,R)\,$ and denote
$$
u_1:=\Phi(v_1)\quad\hbox{and}\quad u_2:=\Phi(v_2).$$

Let  $X_1$ and $X_2$ be the flows associated to $v_1$ and $v_2.$
Set $A_i=(DX_i)^{-1}$ and $J_i:=\det DX_i$ for $i=1,2.$ 
The equation satisfied by 
$\du:=u_2-u_1$  reads 
$$
\d_t\du-\mu\Delta\du-\mu'\nabla\div\du=\df:=\df_1+\div\df_2+2\mu\div\df_3+\lambda\div\df_4
$$
with 
$\df_1:=-a_0\d_t\du,$
$$\displaylines{
\df_2:=\adj(DX_1)P(\rho_0J_1^{-1})-\adj(DX_2)P(\rho_0J_2^{-1}),\cr
\df_3:=\adj(DX_2)D_{A_2}(u_2)- \adj(DX_1)D_{A_1}(u_1)-D(\du),\cr
\df_4:=\adj(DX_2){}^T\!A_2:\nabla u_2-\adj(DX_1){}^T\!A_1:\nabla u_1-\div\du\,\Id.}
$$

Once again, bounding $\du$ in $E_p(T)$  stems from Proposition \ref{p:lamewhole}, which 
ensures that
\begin{equation}\label{eq:dU}
\|\du\|_{E_p(T)}\lesssim\|\df_1\|_{L_T^1(\dot B^{n/p-1}_{p,1})}
+T\|\df_2\|_{L^\infty_T(\dot B^{n/p}_{p,1})}
+\|\df_3\|_{L_T^1(\dot B^{n/p}_{p,1})}+\|\df_4\|_{L_T^1(\dot B^{n/p}_{p,1})}.
\end{equation}

In order to bound $\df_1,$ we just have to use the definition of the multiplier space $\cM(\dot B^{n/p-1}_{p,1}).$ We get
\begin{eqnarray}\label{eq:df1}
&&\|\df_1\|_{L_T^1(\dot B^{n/p-1}_{p,1})}\leq \|a_0\|_{\cM(\dot B^{n/p-1}_{p,1})}\|\d_t\du\|_{L^1_T(\dot B^{n/p-1}_{p,1})}.
\end{eqnarray}
Next,  using the decomposition 
$$
\df_2=(\adj(DX_1)-\adj(DX_2))P(\rho_0J_2^{-1})+\adj(DX_1)(P(\rho_0J_1^{-1})-P(\rho_0J_2^{-1})),
$$
together with composition inequalities  \eqref{eq:compo1}, \eqref{eq:compo2} and \eqref{eq:dAdj},  and product laws in Besov space yields
\begin{equation}\label{eq:df3}
\|\df_2\|_{L_T^\infty(\dot B^{n/p}_{p,1})}\lesssim
T(1+\|a_0\|_{\dot B^{n/p}_{p,1}}) \|D\dv\|_{L_T^1(\dot B^{n/p}_{p,1})}
 \end{equation}
 Finally, we have
 $$
 \df_4=(\adj(DX_2)-\adj(DX_1)){}^T\!A_2:\nabla u_2+\adj(DX_1){}^T\!(A_2-A_1):\nabla u_2
 +(\adj(DX_1){}^TA_1-\Id):\nabla\du,
 $$
 whence, by virtue of  \eqref{eq:U1}, \eqref{eq:U2},  \eqref{eq:dA} and  \eqref{eq:dAdj},
  \begin{equation}\label{eq:df4}
\|\df_4\|_{L_T^1(\dot B^{n/p}_{p,1})}\lesssim
 \|D\dv\|_{L_T^1(\dot B^{n/p}_{p,1})}
 \|Du_2\|_{L_T^1(\dot B^{n/p}_{p,1})}+\|D\du\|_{L_T^1(\dot B^{n/p}_{p,1})}
 \|Dv_1\|_{L_T^1(\dot B^{n/p}_{p,1})}.
\end{equation}
 Bounding $\df_3$ works exactly the same.   So we see that if Conditions \eqref{eq:smallT} and  \eqref{eq:smallness1}   are  satisfied 
 (with smaller $\eta$ and larger $C$ if need be) then we have  
 $$
  \|\du\|_{E_p(T)}\leq \frac12\|\dv\|_{E_p(T)}.
  $$
  Hence, the map $\Phi: \bar B_{E_p(T)}(u_L,R)\mapsto   \bar B_{E_p(T)}(u_L,R)  $ is $1/2$-Lipschitz.
  Therefore, Banach' fixed point theorem ensures that $\Phi$ admits a unique fixed point in $\bar B_{E_p(T)}(u_L,R).$
   This completes the proof of existence of a  solution in $E_p(T)$ for System \eqref{eq:lagrangian}. 
   
   A tiny variation over the proof of the contraction properties yields uniqueness and Lipschitz continuity
   of the flow map.  We eventually get:
\begin{theo}\label{th:main1lagrangian}
Assume that $n\geq1.$  Let $p\in[1,\infty)$ and  $u_0$ be a vector-field in $\dot B^{n/p-1}_{p,1}.$
 Assume that the initial density $\rho_0$ satisfies  $a_0:=(\rho_0-1)\in \dot B^{n/p}_{p,1}.$
 There exists a constant $c$ depending only on $p$ and on $n$ such that if
 \begin{equation}\label{eq:smallrho}
 \|a_0\|_{\cM(\dot B^{n/p-1}_{p,1})}\leq c
 \end{equation}
 then System \eqref{eq:lagrangian} has a unique local  solution $(\bar\rho,\bar u)$ with 
$(\bar a,\bar u)\in\cC([0,T];\dot B^{n/p}_{p,1})\times E_p(T).$
Moreover, 
 the flow map $(a_0,u_0)\longmapsto (\bar a,\bar u)$
is Lipschitz  continuous from  $B^{n/p}_{p,1}\times\dot B^{n/p-1}_{p,1}$
to  $\cC([0,T];\dot B^{n/p}_{p,1})\times E_p(T).$
 \end{theo}  In Eulerian coordinates, this result recasts in:
\begin{theo}\label{th:main1euler}
Under the above assumptions with in addition $n\geq2$ and $p<2n,$  System \eqref{eq:euler} has a unique local solution $(\rho, u)$ with
density bounded away from vacuum and  
$a\in \cC([0,T];\dot B^{n/p-1}_{p,1})$ and $u\in E_p(T).$
 \end{theo}
   We do not give here more details on how to complete the proof of Theorem \ref{th:main1lagrangian}
   and its Eulerian counterpart, Theorem \ref{th:main1euler}, as it will done in the next section under much more general assumptions.

%%%%%%%%%%%%%%%%%%%%%%%%%%%%%%%%%%%%%%%%%%%%

\section{The fully nonhomogeneous case}\label{s:nonhomo}

For treating the general case where $\rho_0$ only  satisfies \eqref{eq:positive},
just resorting to Proposition \ref{p:lamewhole} is not enough because the term $I_1(v,v)$ in the r.h.s. of \eqref{eq:baru} need not be small. 
One has first to establish a similar statement 
for a Lam\'e system with \emph{nonconstant} coefficients. 
More precisely, keeping in mind that $\rho=J_u^{-1}\rho_0$  (we still drop the bars for notational simplicity), 
 we recast the velocity equation of \eqref{eq:lagrangian}  in:
$$
L_{\rho_0}( u)= \rho_0^{-1}\div\bigl(I_1(u, u)+I_2(u, u)
+I_3( u, u)+I_4(u)\bigr)
$$
 with 
 \begin{equation}\label{eq:L}
 L_{\rho_0}(u):=\d_t u-\rho_0^{-1}\div\bigl(2\mu(\rho_0)D(u)+\lambda(\rho_0)\div u\,\Id\bigr)
 \end{equation}
 and 
 $$
 \begin{array}{lll}
 %I_1( v, w)&\!\!\!:=\!\!\!&(1-J_v)\d_t w\\[1.5ex]
 I_1(v,w)&\!\!\!:=\!\!\!&(\adj(DX_v)-\Id)\bigl(\mu(J_v^{-1}\rho_0)(D w\, A_v+{}^T\!A_v\,\nabla w)
+\lambda(J_v^{-1}\rho_0)({}^T\!A_v:\nabla w)\Id\bigr)\\[1.5ex]
 I_2(v,w)&\!\!\!:=\!\!\!&(\mu(J_v^{-1}\rho_0)-\mu(\rho_0))(D w\, A_v+{}^T\!A_v\,\nabla  w)
 +(\lambda(J_v^{-1}\rho_0)-\lambda(\rho_0))({}^T\!A_v:\nabla w)\Id\\[1.5ex]
 I_3(v,w)&\!\!\!:=\!\!\!&\mu(\rho_0)\bigl(D w(A_v-\Id)+{}^T\!(A_v-\Id)\nabla w\bigr)
 +\lambda(\rho_0)({}^T\!(A_v-\Id):\nabla w)\Id \\[1.5ex]
 I_4(v)&\!\!\!:=\!\!\!&-\adj(DX_v)P(\rho_0J_v^{-1}).
 \end{array}
$$

Therefore, in order to solve \eqref{eq:lagrangian} locally, it suffices to show that the map 
\begin{equation}\label{eq:Phi}
\Phi: v\longmapsto  u
\end{equation}
with $u$ the solution to 
$$\left\{\begin{array}{l}
L_{\rho_0}(u)= \rho_0^{-1}\div\bigl(I_2( v, v)+I_3( v, v)
+I_4(v,v)+I_5(v)\bigr),\\[1.5ex]
u|_{t=0}=u_0\end{array}\right.$$
 has a fixed point in $E_p(T)$ for small enough $T.$ 
 \medbreak
 
 As a first step, we have to study the properties of the linear Lam\'e operator $L_{\rho_0}.$
 This is done in the following subsection.
 
%%%%%%%%%%%%%%%%%%%%%%%%%%%%%%%%

\subsection{Linear parabolic systems with rough coefficients}\label{s:linear}

As a warm up, we consider the following scalar heat equation
with variable coefficients:
\begin{equation}\label{eq:heatvar}
\d_tu-a\div(b\nabla u)=f.
\end{equation}
We assume that 
\begin{equation}\label{eq:nonzero}
\alpha:=\inf_{(t,x)\in[0,T]\times \R^n} (ab)(t,x)>0.
\end{equation} 
Let us first consider the smooth case.
\begin{prop}\label{p:heatsmooth} 
Assume that $a$ and $b$ are bounded  functions satisfying \eqref{eq:nonzero}
and such that
$b\nabla a$ and $a\nabla b$ are in $L^2(0,T;\dot B^{n/p}_{p,1})$ for some  $1<p<\infty.$
There exist two  constants $\kappa=\kappa(p)$  and $C=C(s,n,p)$ such that
 the solutions to \eqref{eq:heatvar} satisfy for all $t\in[0,T],$
$$
\|u\|_{L^\infty_t(\dot B^s_{p,1})}+\kappa\alpha\|u\|_{L^1_t(\dot B^{s+2}_{p,1})}
\leq \bigl(\|u_0\|_{\dot B^s_{p,1}}+\|f\|_{L_t^1(\dot B^s_{p,1})}\bigr)
\exp\biggl(\frac C{\alpha}\int_0^t\|(b\nabla a,a\nabla b)\|^2_{\dot B^{n/p}_{p,1}}\,d\tau\biggr)
$$
whenever $-\min(n/p,n/p')<s\leq n/p.$
\end{prop}
\begin{proof}
We first rewrite the equation for $u$ as follows:
$$
\d_tu-\div(ab\nabla u)=f-b\nabla a\cdot\nabla u,
$$
then localize the equation in the Fourier space, according to Littlewood-Paley decomposition:
$$
\d_tu_j-\div(ab\nabla u_j)=f_j-\ddj(b\nabla a\cdot\nabla u)+R_j
$$
with
$u_j:=\ddj u,$ $f_j:=\ddj f$ and 
$R_j:=\div([\ddj,ab]\nabla u).$
\medbreak
Next, we multiply the above equation by $u_j|u_j|^{p-2}$ and integrate over $\R^n.$
Taking advantage of Lemma 8 in the appendix of  \cite{D5}  (here $1<p<\infty$ comes into play) and of
H\"older inequality, we get for some constant $c_p$ depending only on $p$:
$$
\frac1p\frac d{dt}\|u_j\|_{L^p}^p+c_p\alpha2^{2j}\|u_j\|_{L^p}^p
\leq \|u_j\|_{L^p}^{p-1}\bigl(\|f_j\|_{L^p}+\|\ddj(b\nabla a\cdot\nabla u)\|_{L^p}+\|R_j\|_{L^p}\bigr),
$$
which, after time integration, leads to 
\begin{multline}\label{eq:uj}
\|u_j\|_{L^\infty_t(L^p)}+c_p\alpha2^{2j}\|u_j\|_{L^1_t(L^p)}\leq \|u_{0,j}\|_{L^p}\\+\|f_j\|_{L^1_t(L^p)}
+\int_0^t\Bigl(\|\ddj(b\nabla a\cdot\nabla u)\|_{L^p}+\|R_j\|_{L^p}\Bigr)\,d\tau.
\end{multline}
According to Lemmas \ref{l:prod} and \ref{l:com} in Appendix, there exist a positive constant $C$ and  some sequence $(c_j)_{j\in\Z}$ with
$\|c\|_{\ell^1(\Z)}=1,$ satisfying
\begin{equation}\label{eq:com1}
\|\ddj(b\nabla a\cdot\nabla u)\|_{L^p}+\|R_j\|_{L^p}\leq  Cc_j2^{-js}\bigl(\|b\nabla a\|_{\dot B^{n/p}_{p,1}}
+\|a\nabla b\|_{\dot B^{n/p}_{p,1}}\bigr)\|\nabla u\|_{\dot B^s_{p,1}}.
\end{equation}
Then inserting \eqref{eq:com1} in \eqref{eq:uj}, multiplying by $2^{js}$ and summing up over $j$ yields
\begin{multline}\label{eq:uj2}
\|u\|_{L^\infty_t(\dot B^s_{p,1})}+c_p\alpha \|u\|_{L^1_t(\dot B^{s+2}_{p,1})}
\leq \|u_0\|_{\dot B^s_{p,1}}+\|f\|_{L^1_t(\dot B^s_{p,1})}\\
+C\int_0^t \|(b\nabla a,a\nabla b)\|_{\dot B^{n/p}_{p,1}}\|u\|_{\dot B^{s+1}_{p,1}}\,d\tau.
\end{multline}
{}From  the interpolation inequality
\begin{equation}\label{eq:interpo}
\|u\|_{\dot B^{s+1}_{p,1}}\leq \|u\|_{\dot B^{s}_{p,1}}^{1/2}\|u\|_{\dot B^{s+2}_{p,1}}^{1/2},
\end{equation}
we gather that 
$$
C \|(b\nabla a,a\nabla b)\|_{\dot B^{n/p}_{p,1}}\|u\|_{\dot B^{s+1}_{p,1}}\leq \frac{\alpha c_p}2\|u\|_{\dot B^{s+2}_{p,1}}
+\frac{C^2}{2\alpha c_p} \|(b\nabla a,a\nabla b)\|_{\dot B^{n/p}_{p,1}}^2\|u\|_{\dot B^s_{p,1}}.
$$
So plugging this in \eqref{eq:uj2} and   applying Gronwall lemma completes the proof of the proposition.
\end{proof}
In the rough case where the coefficients are only in $\dot B^{n/p}_{p,1},$ the above proposition has to 
be modified as follows:
\begin{prop}\label{p:heatrough} 
Let  $a$ and $b$ be  bounded positive and  satisfy \eqref{eq:nonzero}. Assume that
$b\nabla a$ and $a\nabla b$ are in $L^\infty(0,T;\dot B^{n/p-1}_{p,1})$ with $1<p<\infty.$
There exist three constants   $\eta,$  $\kappa$ and $C$  such that if for some $m\in\Z$ we have 
\begin{eqnarray}\label{eq:positiva}
&\inf_{(t,x)\in[0,T]\times\R^n} \dot S_m(ab)(t,x)\geq \alpha/2,\\
\label{eq:smalla}
&\|(\Id-\dot S_m)(b\nabla a,a\nabla b)\|_{L^\infty_T(\dot B^{n/p-1}_{p,1})}\leq \eta\alpha
\end{eqnarray}
then the solution to \eqref{eq:heatvar} satisfies for all $t\in[0,T],$
$$
\|u\|_{L^\infty_t(\dot B^s_{p,1})}+\alpha\kappa\|u\|_{L^1_t(\dot B^{s+2}_{p,1})}
\leq \bigl(\|u_0\|_{\dot B^s_{p,1}}+\|f\|_{L_t^1(\dot B^s_{p,1})}\bigr)
\exp\biggl(\frac C{\alpha}\int_0^t\|\dot S_m(b\nabla a,a\nabla b)\|^2_{\dot B^{n/p}_{p,1}}\,d\tau\biggr)
$$
whenever 
\begin{equation}\label{eq:s}
-\min(n/p,n/p')<s\leq n/p-1.
\end{equation}
\end{prop}
\begin{proof}
Given the new assumptions, it is natural to replace \eqref{eq:com1} by the inequality
\begin{equation}\label{eq:com2}
\|\ddj(b\nabla a\cdot\nabla u)\|_{L^p}+\|R_j\|_{L^p}\leq  Cc_j2^{-js}\bigl(\|b\nabla a\|_{\dot B^{n/p-1}_{p,1}}
+\|a\nabla b\|_{\dot B^{n/p-1}_{p,1}}\bigr)\|\nabla u\|_{\dot B^{s+1}_{p,1}},
\end{equation}
which may be obtained by taking $\sigma=1$ and $\nu=1$ in Lemmas \ref{l:prod} and \ref{l:com}.
However, when bounding $R_j$,  in addition to \eqref{eq:s}, one has to assume that $p\leq n.$
Also, as it involves the highest regularity of $u,$ we cannot expect to absorb this ``remainder term''  any longer, unless $a\nabla b$ and $b\nabla a$
are {\it small} in $\dot B^{n/p-1}_{p,1}$ (which would correspond to the case that has been treated in the previous section).
So  we rather rewrite the heat equation as follows:
$$
\d_tu-\div(\dot S_m(ab)\nabla u)=f+\div((\Id-\dot S_m)(ab)\nabla u)-\dot S_m(b\nabla a)\cdot\nabla u-(\Id-\dot S_m)(b\nabla a)\cdot\nabla u.
$$
Now, using the infimum bound for $\dot S_m(ab)$ and arguing as for proving \eqref{eq:uj},  we get
$$\displaylines{
\|u_j\|_{L^\infty_t(L^p)}+c_p\alpha2^{2j}\|u_j\|_{L^1_t(L^p)}\leq \|u_{0,j}\|_{L^p}+\|f_j\|_{L^1_t(L^p)}
+\int_0^t\|\ddj\div((\Id-\dot S_m)(ab)\nabla u)\|_{L^p}\,d\tau
\hfill\cr\hfill+\int_0^t\Bigl(\|\ddj(\dot S_m(b\nabla a)\cdot\nabla u)\|_{L^p}
+\|\ddj((\Id-\dot S_m)(b\nabla a)\cdot\nabla u)\|_{L^p}+
\|\div([\dot S_m(ab),\ddj]\nabla u)\|_{L^p}\Bigr)\,d\tau.}
$$
The idea is to apply the procedure of the ``smooth'' case for the low frequency part of the coefficients (that is the part containing $\dot S_m$)
and the ``perturbation'' approach for the other part. 
More precisely, appealing to  Lemmas \ref{l:prod} and \ref{l:com},  we get
under Condition \eqref{eq:s} and  for some sequence $(c_j)_{j\in\Z}$ with $\|c\|_{\ell^1(\Z)}=1$:
$$\begin{array}{lll}
\|\ddj\div((\Id-\dot S_m)(ab)\nabla u)\|_{L^p}&\!\!\!\lesssim\!\!\!&c_j2^{-js}
\|(\Id-\dot S_m)(ab)\|_{\dot B^{n/p}_{p,1}}\|\nabla u\|_{\dot B^{s+1}_{p,1}},\\[1ex]
\|\ddj(\dot S_m(b\nabla a)\cdot\nabla u)\|_{L^p}&\!\!\!\lesssim\!\!\!&c_j2^{-js}
\|\dot S_m(b\nabla a)\|_{\dot B^{n/p}_{p,1}}\|\nabla u\|_{\dot B^{s}_{p,1}},\\[1ex]
\|\ddj((\Id-\dot S_m)(b\nabla a)\cdot\nabla u)\|_{L^p}&\!\!\!\lesssim\!\!\!& c_j2^{-js}
\|(\Id-\dot S_m)(b\nabla a)\|_{\dot B^{n/p-1}_{p,1}}\|\nabla u\|_{\dot B^{s+1}_{p,1}},\\[1ex]
\|\div([\dot S_m(ab),\ddj]\nabla u)\|_{L^p}&\!\!\!\lesssim\!\!\!&c_j2^{-js}
\|\dot S_m\nabla(ab)\|_{\dot B^{n/p}_{p,1}}\|\nabla u\|_{\dot B^{s}_{p,1}}.
\end{array}
$$
Let us plug those four inequalities in the above inequality for $u_j.$ After multiplying 
by $2^{js}$ and summing up over $j,$ we get
$$
\displaylines{\|u\|_{L^\infty_t(\dot B^s_{p,1})}+c_p\alpha \|u\|_{L^1_t(\dot B^{s+2}_{p,1})}
\leq \|u_0\|_{\dot B^s_{p,1}}+\|f\|_{L^1_t(\dot B^s_{p,1})}\hfill\cr\hfill
+C\bigl(\|(\Id-\dot S_m)(ab)\|_{L^\infty_t(\dot B^{n/p}_{p,1})}
+\|(\Id-\dot S_m)(b\nabla a)\|_{L^\infty_t(\dot B^{n/p-1}_{p,1})}\bigr)\|u\|_{L^1_t(\dot B^{s+2}_{p,1})}
\hfill\cr\hfill+C\int_0^t\|\dot S_m(a\nabla b,b\nabla a)\|_{\dot B^{n/p}_{p,1}}\|\nabla u\|_{\dot B^{s}_{p,1}}\,d\tau.}
$$
It is clear that, under Condition \eqref{eq:smalla}, the second line may be absorbed by the left-hand side. Hence  the desired inequality follows 
from the interpolation inequality \eqref{eq:interpo}, exactly as in the smooth case.
\end{proof}

%%%%%%%%%%%%%%%%%%%%%%%%%%%%%%%%%%%%%%%
We now look at the following Lam\'e system with nonconstant coefficients:
\begin{equation}\label{eq:lame-var}
\d_tu-2a\div(\mu D(u))-b\nabla(\lambda\div u)=f.
\end{equation}
Note that $u$ and $f$ are valued in $\R^n.$
We assume throughout that the following uniform ellipticity condition is satisfied: 
\begin{equation}\label{eq:ellipticity}
\alpha:=\min\Bigl(\inf_{(t,x)\in[0,T]\times\R^n} (a\mu)(t,x),\inf_{(t,x)\in[0,T]\times\R^n}
(2a\mu+b\lambda)(t,x)\Bigr)>0.
\end{equation}
Let us first study the  ``smooth case'':
\begin{prop}
\label{p:lamesmooth} 
Assume that $a,$ $b,$ $\lambda$ and $\mu$  are bounded  functions satisfying \eqref{eq:ellipticity}
and such that $a\nabla\mu,$ $b\nabla\lambda,$
$\mu\nabla a$ and $\lambda\nabla b$ are in $L^2(0,T;\dot B^{n/p}_{p,1})$ for some  $1<p<\infty.$
There exists a constant $C$ such that
 the solutions to \eqref{eq:lame-var} satisfy for all $t\in[0,T],$
$$\displaylines{\quad
\|u\|_{L^\infty_t(\dot B^s_{p,1})}+\alpha\|u\|_{L^1_t(\dot B^{s+2}_{p,1})}
\hfill\cr\hfill\leq C \bigl(\|u_0\|_{\dot B^s_{p,1}}+\|f\|_{L_t^1(\dot B^s_{p,1})}\bigr)
\exp\biggl(\frac C{\alpha}\int_0^t\|(\mu\nabla a,a\nabla \mu,\lambda\nabla b,b\nabla\lambda)\|^2_{\dot B^{n/p}_{p,1}}\,d\tau\biggr)\quad}
$$
whenever $-\min(n/p,n/p')<s\leq n/p.$
\end{prop}
\begin{proof}
We introduce the following functions:
$$
d:=|D|^{-1}\div u\quad\hbox{and}\quad
\Omega:=|D|^{-1}\curl u\quad\hbox{with }\  (\curl u)_{ij}:=\d_iu^j-\d_ju^i.
$$
Owing to the use of homogeneous Besov space, and
because the Fourier multipliers $A(D):=|D|^{-1}\div$ and $B(D):=|D|^{-1}\curl$
are of degree $0,$ it is equivalent to estimate
$u$ or $(d,\Omega)$ in $L^\infty_T(\dot B^s_{p,1})\cap L^1_T(\dot B^{s+2}_{p,1}).$
So the basic idea is to show that $d$ and $\Omega$ satisfy heat equations 
similar to \eqref{eq:heatvar}. More precisely, applying $A(D)$ to \eqref{eq:lame-var} yields
\begin{multline}\label{eq:d}
\d_td-(2a\mu+b\lambda)\Delta d
=A(D) (f+2a\nabla\mu\cdot D(u)+b\nabla\lambda\,\div u)\\
+[A(D),a\mu]\Delta u+[A(D),a\mu+b\lambda]\nabla\div u.
\end{multline}
Given Condition \eqref{eq:ellipticity}, we see that arguing exactly as for proving \eqref{eq:uj2} 
and because $A(D)$ maps $\dot B^s_{p,1}$ in itself,
$$\displaylines{
\|d\|_{L^\infty_t(\dot B^s_{p,1})}+\kappa\alpha \|d\|_{L^1_t(\dot B^{s\!+\!2}_{p,1})}
\!\leq\! \|d_0\|_{\dot B^s_{p,1}}+\!\|A(D)f\|_{L^1_t(\dot B^s_{p,1})}
+C\!\!\int_0^t\!\|2a\nabla\mu\cdot D(u)+b\nabla\lambda\div u\|_{\dot B^s_{p,1}}d\tau\hfill\cr\hfill
+C\int_0^t\|[A(D),a\mu]\Delta u+[A(D),a\mu+b\lambda]\nabla\div u\|_{\dot B^s_{p,1}}\,d\tau
+C\int_0^t \|\nabla(2a\mu+b\lambda)\|_{\dot B^{n/p}_{p,1}}\|\nabla u\|_{\dot B^s_{p,1}}\,d\tau.}
$$
Note that  applying Lemma \ref{l:com2} with $\sigma=s-1,$ $\nu=0$ and
Lemma \ref{l:prod} with $\sigma=s$ and $\nu=0$ yields
$$
\begin{array}{lll}
\|[A(D),a\mu]\Delta u\|_{\dot B^s_{p,1}}&\leq& C\|\nabla(a\mu)\|_{\dot B^{n/p}_{p,1}}\|\Delta u\|_{\dot B^{s-1}_{p,1}},\\[1ex]
\|a\nabla\mu\cdot D(u)\|_{\dot B^s_{p,1}}&\leq&C\|a\nabla\mu\|_{\dot B^{n/p}_{p,1}}\|\nabla u\|_{\dot B^s_{p,1}},
\end{array}
$$
and analogous  estimates for 
$[A(D),a\mu+b\lambda]\nabla\div u$ and $b\nabla\lambda\div u.$
\smallbreak
Similarly, the vorticity part $\Omega$ of $u$ satisfies
$$
\displaylines{\d_t\Omega-a\mu\Delta\Omega
=B(D) (f+2a\nabla\mu\cdot D(u)+b\nabla\lambda\div u)
+[B(D),a\mu]\Delta u+[B(D),a\mu+b\lambda]\nabla\div u.}
$$
So arguing exactly as for bounding $d,$ and resorting to 
the interpolation inequality  \eqref{eq:interpo} and to Gronwall lemma, 
we easily get the desired inequality. It is just a matter of following the proof for the case of the heat equation.
\end{proof}

Let us finally focus on the ``rough  case'' where the coefficients of \eqref{eq:lame-var} are only in $L_T^\infty(\dot B^{n/p}_{p,1}).$ 
\begin{prop}\label{p:lamerough} 
Let  $a,$  $b,$ $\lambda$ and $\mu$  be  bounded functions  satisfying  \eqref{eq:ellipticity}. Assume that
$a\nabla \mu,$ $b\nabla\lambda,$ $\mu\nabla a$ and $\lambda\nabla b$ are in $L^\infty(0,T;\dot B^{n/p-1}_{p,1})$ for some $1<p<\infty.$
There exist two constants   $\eta$ and $\kappa$  such that if for some $m\in\Z$ we have 
\begin{eqnarray}\label{eq:positivebis}
&\min\Bigl(\inf_{(t,x)\in[0,T]\times\R^n} \dot S_m(2a\mu+b\lambda)(t,x),
\inf_{(t,x)\in[0,T]\times\R^n} \dot S_m(a\mu)(t,x)\Bigr)
\geq\Frac\alpha2,\\
\label{eq:smallabis}
&\|(\Id-\dot S_m)(\mu\nabla a,a\nabla \mu,\lambda\nabla b,b\nabla\lambda)\|_{L^\infty_T(\dot B^{n/p-1}_{p,1})}\leq \eta\alpha
\end{eqnarray}
then the solutions to \eqref{eq:lame-var} satisfy for all $t\in[0,T],$
$$\displaylines{\quad
\|u\|_{L^\infty_t(\dot B^s_{p,1})}+\alpha\|u\|_{L^1_t(\dot B^{s+2}_{p,1})}
\hfill\cr\hfill\leq C \bigl(\|u_0\|_{\dot B^s_{p,1}}+\|f\|_{L_t^1(\dot B^s_{p,1})}\bigr)
\exp\biggl(\frac C{\alpha}\int_0^t\|\dot S_m(\mu\nabla a,a\nabla \mu,\lambda\nabla b,b\nabla\lambda)\|^2_{\dot B^{n/p}_{p,1}}\,d\tau\biggr)\quad}
$$
whenever $-\min(n/p,n/p')<s\leq n/p-1.$
\end{prop}
\begin{proof}
As for the heat equation, we split the coefficients of the system 
into a smooth (but large) low frequency part  and a rough (but small) high frequency part. 
It turns out to be  more convenient  to work  directly on the equations for $d$ and $\Omega.$
More precisely, as regards $d,$ we write (starting from \eqref{eq:d} and denoting $c:=2a\mu+b\lambda$) that 
$$
\displaylines{\d_td-\div(c\nabla d)
=-\nabla c\cdot\nabla d+A(D) (f+2a\nabla\mu\cdot D(u)+b\nabla\lambda\,\div u)\hfill\cr\hfill
+[A(D),a\mu]\Delta u+[A(D),a\mu+b\lambda]\nabla\div u,}
$$ 
whence, denoting $d_j:=\ddj d,$
$$
\displaylines{\d_td_j-\div(\dot S_m c\nabla d_j)
=\div([\ddj,\dot S_mc]\nabla d)\hfill\cr\hfill+\ddj\Bigl(\div((\Id-\dot S_m)c\nabla d)
-\dot S_m\nabla c\cdot\nabla d-(\Id-\dot S_m)\nabla c\cdot\nabla d\Bigr)
\hfill\cr\hfill
 +\ddj A(D) \bigl(f+2\dot S_m(a\nabla\mu)\cdot D(u)+2(\Id-\dot S_m)(a\nabla\mu)\cdot D(u)
  \hfill\cr\hfill +\dot S_m(b\nabla\lambda)\,\div u
   +(\Id-\dot S_m)(b\nabla\lambda)\,\div u\bigr)
+\ddj\Bigl([A(D),\dot S_m(a\mu)]\Delta u  \hfill\cr\hfill+[A(D),\dot S_m(a\mu+b\lambda)]\nabla\div u
+[A(D),(\Id-\dot S_m)(a\mu)]\Delta u+[A(D),(\Id-\dot S_m)(a\mu+b\lambda)]\nabla\div u\Bigr).}
$$
Under Condition \eqref{eq:s}, Lemmas \ref{l:prod}, \ref{l:com} and \ref{l:com2} imply that 
$$
\begin{array}{lll}
\|\div([\ddj,\dot S_mc]\nabla d)\|_{L^p}&\lesssim& c_j2^{-js}\|\dot S_m\nabla c\|_{\dot B^{n/p}_{p,1}}\|\nabla d\|_{\dot B^s_{p,1}},\\[1ex]
\|\ddj\div((\Id-\dot S_m)c\nabla d)\|_{L^p}&\lesssim& c_j2^{-js}\|(\Id-\dot S_m)c\|_{\dot B^{n/p}_{p,1}}\|\nabla d\|_{\dot B^{s+1}_{p,1}},\\[1ex]
\|\ddj(\dot S_m\nabla c\cdot\nabla d)\|_{L^p}&\lesssim& c_j2^{-js}\|\dot S_m\nabla c\|_{\dot B^{n/p}_{p,1}}\|\nabla d\|_{\dot B^s_{p,1}},\\[1ex]
\|\ddj((\Id-\dot S_m)\nabla c\cdot\nabla d)&\lesssim& c_j2^{-js}\|(\Id-\dot S_m)\nabla c\|_{\dot B^{n/p-1}_{p,1}}\|\nabla d\|_{\dot B^{s+1}_{p,1}},\\[1ex]
\|\ddj [A(D),\dot S_m(a\mu)]\Delta u\|_{L^p}&\lesssim& c_j2^{-js}\|\nabla\dot S_m(a\mu)\|_{\dot B^{n/p}_{p,1}}\|\Delta u\|_{\dot B^{s-1}_{p,1}},\\[1ex]
\|\ddj [A(D),(\Id-\dot S_m)(a\mu)]\Delta u\|_{L^p}&\lesssim& c_j2^{-js}\|\nabla(\Id-\dot S_m)(a\mu)\|_{\dot B^{n/p-1}_{p,1}}\|\Delta u\|_{\dot B^{s}_{p,1}},
\end{array}
$$
and similar estimates for 
$$\displaylines{
\ddj A(D)(\dot S_m(a\nabla\mu)\cdot D(u)),\quad \ddj A(D)((\Id-\dot S_m)(a\nabla\mu)\cdot D(u)),\cr
\ddj A(D)(\dot S_m(b\nabla\lambda)\,\div u),\quad
\ddj A(D)((\Id-\dot S_m)(b\nabla\lambda)\,\div u),\cr
\ddj [A(D),\dot S_m(a\mu+b\lambda)]\nabla\div u,\quad
\ddj[A(D),(\Id-\dot S_m)(a\mu+b\lambda)]\nabla\div u.}$$
The $\curl$ part $\Omega$ of the velocity may be treated in the same way. 
Therefore we get 
$$\displaylines{\quad
\|u\|_{L_t^\infty(\dot B^s_{p,1})}+\alpha \|u\|_{L_t^1(\dot B^{s+2}_{p,1})}\lesssim
\|u_0\|_{\dot B^s_{p,1}}+\|f\|_{L^1_t(\dot B^s_{p,1})}\hfill\cr\hfill+\int_0^t\|\dot S_m(a\nabla\mu,\mu\nabla a,b\nabla\lambda,\lambda\nabla b)\|_{\dot B^{n/p}_{p,1}}\|u\|_{\dot B^{s+1}_{p,1}}\,d\tau\hfill\cr\hfill
+\int_0^t\|(\Id-\dot S_m)(a\nabla\mu,\mu\nabla a,b\nabla\lambda,\lambda\nabla b)\|_{\dot B^{n/p-1}_{p,1}}\|u\|_{\dot B^{s+2}_{p,1}}\,d\tau.\quad}
$$
Obviously the last term may be absorbed by the left-hand side if $\eta$ is small enough in \eqref{eq:smallabis}
and the last-but-one term may be handled by interpolation according to \eqref{eq:interpo}.
So applying Gronwall lemma yields the desired inequality. 
\end{proof}

For the sake of completeness, we still have to justify the existence of a solution to \eqref{eq:lame-var}.
More precisely, we want to establish the following result:
\begin{prop}\label{p:lameexistence} 
Let $p$ be in $(1,+\infty).$ 
Let  $a,$  $b,$ $\lambda$ and $\mu$  be  bounded functions  satisfying  \eqref{eq:ellipticity}. 
Assume in addition that there exist 
some constants $\bar a,$
$\bar b,$ $\bar\lambda$ and $\bar\mu$ such that
\begin{equation}
2\bar a\bar\mu+\bar b\bar\lambda>0\quad\hbox{and}\quad\bar a\bar\mu>0,
\end{equation}
and such that $a-\bar a,$ $b-\bar b,$ $\mu-\bar\mu$ and $\lambda-\bar\lambda$
are in $\cC([0,T];\dot B^{n/p}_{p,1}).$
Finally, suppose that
\begin{equation}\label{eq:smallcoeff}
\lim_{m\rightarrow+\infty} \|(\Id-\dot S_m)(a-\bar a,b-\bar b,\lambda-\bar\lambda,\mu-\bar\mu)\|_{L_T^\infty(\dot B^{n/p}_{p,1})}=0.
\end{equation}
Then for any data $u_0\in \dot B^s_{p,1}$ and
$f\in L^1(0,T;\dot B^s_{p,1})$ with  $s$ satisfying \eqref{eq:s}, System 
 \eqref{eq:lame-var} admits a unique solution 
$u\in \cC([0,T];\dot B^s_{p,1})\cap L^1(0,T;\dot B^{s+2}_{2,1}).$
Besides,  the estimates of Proposition \ref{p:lamerough} are fulfilled 
for all large enough  $m\in\Z.$ 
\end{prop}
\begin{proof}
The proof is based on the continuity method as explained in e.g. \cite{Krylov}
(and used in \cite{D4} in a similar context as ours). 
For $\theta\in[0,1],$ we introduce the following second order operator 
$\cP_\theta$ acting on vector-fields $u$ as follows: 
$$\cP_\theta u:=-2a_\theta\div(\mu_\theta D(u))-b_\theta\nabla(\lambda_\theta\div u),
$$
where
$a_\theta:=(1-\theta)\bar a + \theta a,$ $b_\theta:=(1-\theta)\bar b +\theta b,$ and so on.
We claim that one may find some $m\in\Z$ \emph{independent of $\theta$} such 
that for all $\theta\in[0,1],$ the conditions \eqref{eq:positivebis} and \eqref{eq:smallabis} 
are satisfied by $a_\theta,$ $b_\theta,$ $\mu_\theta$ and $\lambda_\theta.$ 
Indeed, we notice that 
$$
a_\theta-\bar a=\theta(a-\bar a).
$$
Hence, for all $\theta\in[0,1],$
$$
\|(\Id-\dot S_m)(a_\theta-\bar a)\|_{L_T^\infty(\dot B^{n/p}_{p,1})}\leq
\|(\Id-\dot S_m)(a-\bar a)\|_{L_T^\infty(\dot B^{n/p}_{p,1})},
$$
and similar properties hold for $b_\theta,$ $\lambda_\theta$ and $\mu_\theta.$
In particular, owing to the continuous embedding  of $\dot B^{n/p}_{p,1}$ in the
set of continuous bounded functions, and to \eqref{eq:smallcoeff},  we deduce that 
there exists some $m\in\Z$ so that 
the ellipticity condition 
\eqref{eq:positivebis}  is satisfied by operator $\cP_\theta$ for all $\theta\in[0,1].$

Likewise, we have for instance
$$
\mu_\theta\nabla a_\theta=\theta(1-\theta)\bar\mu\nabla a+\theta^2\mu\nabla a
$$
and similar relations for the other coefficients. Hence
one may find some large enough $m$ so that 
 \eqref{eq:smallabis} is satisfied for all $\theta\in[0,1].$ 
 In addition, the above relation shows that 
 $$
\| \dot S_m(\mu_\theta\nabla a_\theta)\|_{\dot B^{n/p}_{p,1}}
\leq \bar\mu\|\dot S_m\nabla a\|_{\dot B^{n/p}_{p,1}}
+\|\dot S_m(\mu\nabla a)\|_{\dot B^{n/p}_{p,1}}.
$$
Hence all the terms appearing in the exponential term of 
 the estimate in Proposition \ref{p:lamerough} may 
 be bounded by a constant \emph{depending only on $m$ and on 
 the coefficients $a,$ $b,$ $\lambda$ and $\mu.$}
 As a conclusion,  one may thus find some constant $C$ \emph{independent of $\theta$}
 such that any solution $w$ of 
 $$
 \d_tw-\cP_\theta w=g,\qquad w|_{t=0}=w_0
 $$
 satisfies 
 \begin{equation}\label{eq:estrough}
 \|w\|_{L_T^\infty(\dot B^s_{p,1})}+\alpha\|w\|_{L_T^1(\dot B^{s+2}_{p,1})}
 \leq C\bigl(\|w_0\|_{\dot B^s_{p,1}}+\|g\|_{L_T^1(\dot B^s_{p,1})}\bigr).
 \end{equation}
 After this preliminary work, one may start with the proof of existence (uniqueness follows from 
 the estimates of Proposition \ref{p:lamerough}). 
Let $\cE$ be  the set
 of those  $\theta$ in $[0,1]$ such that for every data $u_0$ and $f$ (as
in the statement of the theorem), System
  \begin{equation}\label{eq:system}\d_tu-\cP_\theta u=f,\qquad u|_{t=0}=u_0\end{equation}
   has a solution 
$u$ in the set $F^s_p(T):=\cC([0,T];\dot B^s_{p,1})\cap L^1(0,T;\dot B^{s+2}_{p,1}).$   
  
Note that according to Proposition \ref{p:lamewhole}, 
the set $\cE$ contains $0$ hence is nonempty. 
So it suffices to find a \emph{fixed}  $\ep>0$ such that for all $\theta_0\in\cE,$ we have 
\begin{equation}\label{eq:ep}
[\theta_0-\ep,\theta_0+\ep]\cap [0,1]\subset\cE.
\end{equation}
So let us fix some $\theta_0\in\cE,$  $u_0\in \dot B^s_{p,1},$ 
$f\in L^1(0,T;\dot B^s_{p,1})$ and $v\in F_p^s(T)$ 
and consider the solution $u$  to the system
$$
\d_tu-\cP_{\theta_0}u=f+(\cP_{\theta}-\cP_{\theta_0})v
$$
with $\theta\in[0,1]$ such that $|\theta-\theta_0|\leq\ep.$
Given that  $\theta_0$ is in $\cE,$  the existence of $u$ in $F_p^s(T)$
is granted if $(\cP_{\theta}-\cP_{\theta_0})v\in L^1(0,T;\dot B^s_{p,1}).$ 
So let us first check this: we have
$$\displaylines{\quad
(\cP_{\theta}-\cP_{\theta_0})v=(\theta-\theta_0)\Bigl(
2a_{\theta_0}\div\bigl((\bar\mu-\mu)D(v)\bigr)
+2(\bar a-a)\div(\mu_\theta D(v))
\hfill\cr\hfill+b_{\theta_0}\nabla\bigl((\bar\lambda-\lambda)\div v\bigr)
+(\bar b-b)\nabla\bigl(\lambda_\theta\div v\bigr)\Bigr).}
$$
Under Condition \eqref{eq:s}, one may thus conclude thanks to product estimates in Besov spaces 
(see Lemma \ref{l:prod}) that
$(\cP_{\theta}-\cP_{\theta_0})v\in L^1(0,T;\dot B^s_{p,1}).$ Furthermore 
$$\displaylines{\quad
\|(\cP_{\theta}-\cP_{\theta_0})v\|_{\dot B^s_{p,1}}\leq 
C\ep\Bigl((\bar a+\|a_{\theta_0}-\bar a\|_{\dot B^{n/p}_{p,1}})\|\mu-\bar\mu\|_{\dot B^{n/p}_{p,1}}
+ (\bar \mu+\|\mu_{\theta}-\bar \mu\|_{\dot B^{n/p}_{p,1}})\|a-\bar a\|_{\dot B^{n/p}_{p,1}}
\hfill\cr\hfill
+ (\bar b+\|b_{\theta_0}-\bar b\|_{\dot B^{n/p}_{p,1}})\|\lambda-\bar\lambda\|_{\dot B^{n/p}_{p,1}}
+ (\bar \lambda+\|\lambda_{\theta}-\bar \lambda\|_{\dot B^{n/p}_{p,1}})\|a-\bar a\|_{\dot B^{n/p}_{p,1}}
\Bigr)\|Dv\|_{\dot B^{s+1}_{p,1}}.}
$$
The coefficients may be bounded in terms of the initial coefficients $a,$ $b,$ $\lambda$ and $\mu.$
Hence, applying \eqref{eq:estrough} we get for some constant independent of $\theta_0$ and of $\theta,$
$$
\|u\|_{L_T^\infty(\dot B^s_{p,1})}+\alpha\|u\|_{L_T^1(\dot B^{s+2}_{p,1})} \leq 
C\bigl(\ep\|v\|_{L_T^1(\dot B^{s+2}_{p,1})}+ \|w_0\|_{\dot B^s_{p,1}}+\|f\|_{L_T^1(\dot B^s_{p,1})}\bigr).
$$
Taking $\ep$ small enough, it becomes  clear that the linear map 
$\Psi_\theta:v\mapsto u$ is contractive on the Banach space $F_p^s(T).$
Hence it has a (unique) fixed point $u\in F_p^s(T).$
In other words, $u$ satisfies \eqref{eq:system}. 

Given that $\cE$ is nonempty and that $\ep$ is independent of $\theta_0,$ one may 
now conclude that  $1$ is in $\cE.$ Therefore, there exists a solution  $u\in F_{p}^s(T)$
 to \eqref{eq:lame-var}.
\end{proof}
\begin{rem}\label{r:lame}
Under the assumptions of the above proposition, 
the constructed solution $u$ satisfies $\d_tu\in L^1(0,T;\dot B^s_{p,1}).$
Indeed, it suffices to notice that
$$
\d_tu=f+(\bar a+(a-\bar a))\div(\bar\mu+(\mu-\bar\mu) D(u))
+(\bar b+(b-\bar b))\nabla(\bar\lambda+(\lambda-\bar\lambda)\div u),
$$
and to use Lemma \ref{l:prod} together with the facts
that $\nabla u$ is in $L^1(0,T;\dot B^{s+1}_{p,1}).$
Moreover  we have
$$\displaylines{\quad
\|\d_tu\|_{L^1_T(\dot B^s_{p,1})}\leq C \bigl(\|u_0\|_{\dot B^s_{p,1}}+\|f\|_{L_t^1(\dot B^s_{p,1})}\bigr)
\exp\biggl(\frac C{\alpha}\int_0^t\|\dot S_m(\mu\nabla a,a\nabla \mu,\lambda\nabla b,b\nabla\lambda)\|^2_{\dot B^{n/p}_{p,1}}\,d\tau\biggr)\quad}
$$
where $C$ may depend also on the norm of $a-\bar a,$ $b-\bar b,$ $\lambda-\bar\lambda$
and $\mu-\bar \mu$ in $L^\infty(0,T;\dot B^{n/p}_{p,1}).$
\end{rem}

%%%%%%%%%%%%%%%%%%%%%%%%%%%%%%%%%%%%%%%%%%%

\subsection{Proof of Theorem \ref{th:main1}}

As we want to consider (possibly) large velocities, we introduce, as in the almost 
homogeneous case the free solution to the Lam\'e system  corresponding to $\rho\equiv1,$
 that is the vector-field $u_L$
in $E_p(T),$ given by Proposition \ref{p:lamewhole}, satisfying\footnote{See
\eqref{eq:L} for the definition of operator $L_1.$}
 $$
 L_1 u_L=0,\qquad  u|_{t=0}=u_0.
 $$
 We claim that the Banach fixed point theorem applies to the map $\Phi$ defined in 
 \eqref{eq:Phi} in some closed ball $\bar B_{E_p(T)}(u_L,R)$ with suitably small $T$ and $R.$ 
 Denoting $\tilde u:=u-u_L,$ we see that $\tilde u$ has to satisfy 
\begin{equation}\label{eq:tildeu}
 \left\{\begin{array}{l}
 L_{\rho_0}\tilde u=\rho_0^{-1}\div\bigl(I_1(v,v)+I_2(v,v)
 +I_3(v,v)+I_4(v)\bigr)+(L_1-L_{\rho_0})u_L,\\[1ex]
 \tilde u|_{t=0}=0.\end{array}\right.
 \end{equation}
If the right-hand side is in $L^1(0,T;\dot B^{n/p-1}_{p,1})$
and if there exists some $m\in\Z$ so that 
\eqref{eq:positivebis} and \eqref{eq:smallabis} are fulfilled then 
 Proposition \ref{p:lameexistence} and  Remark \ref{r:lame} ensure 
 the existence of $\tilde u$ in $E_p(T).$
 Now, the existence 
 of $m$  so that 
 $$
\displaylines{\min\biggl(\inf_x\dot S_m\Bigl(2\frac{\mu(\rho_0)}{\rho_0}+\frac{\lambda(\rho_0)}{\rho_0}\Bigr),
\inf_x\dot S_m\Bigl(\frac{\mu(\rho_0)}{\rho_0}\Bigr)\biggr)>\frac\alpha2\cr
\hbox{and}\quad
\Bigl\|(\Id-\dot S_m)\Bigl(\frac{\mu(\rho_0)}{\rho_0^2}\nabla\rho_0,\frac{\mu'(\rho_0)}{\rho_0}\nabla\rho_0,\frac{\lambda(\rho_0)}{\rho_0^2}\nabla\rho_0,\frac{\lambda'(\rho_0)}{\rho_0}\nabla\rho_0\Bigr)\Bigr\|_{\dot B^{n/p-1}_{p,1}}\leq\eta\alpha.}
$$
 is ensured by the fact that all the coefficients (minus some constant) 
 belong to the space $\dot B^{n/p}_{p,1}$  which is defined in terms of  a convergent series
 and embeds continuously in the set of bounded continuous functions.
 The study of the right-hand side of \eqref{eq:tildeu} will
 be carried out below. 

\subsubsection*{First step: Stability of $\bar B_{E_p(T)}(u_L,R)$  for small enough $R$ and $T$}
 Proposition \ref{p:lamerough} 
and the definition of the multiplier space $\cM(\dot B^{n/p-1}_{p,1})$ 
 ensure that
 \begin{multline}\label{eq:ut1}
 \|\tilde u\|_{E_p(T)}\leq Ce^{C_{\rho_0,m}T}
\Bigl(\|(L_1-L_{\rho_0})u_L\|_{L_T^1(\dot B^{n/p-1}_{p,1})}\\
+\|\rho_0^{-1}\|_{\cM(\dot B^{n/p-1}_{p,1})}\bigl(\|I_1(v,v)\|_{L_T^1(\dot B^{n/p}_{p,1})}+
 \|I_2(v,v)\|_{L_T^1(\dot B^{n/p}_{p,1})}+\|I_3(v,v)\|_{L_T^1(\dot B^{n/p}_{p,1})}
 \!+\|I_4(v)\|_{L_T^1(\dot B^{n/p}_{p,1})}\bigr)\!\Bigr)
 \end{multline}
 for some constant $C_{\rho_0,m}$ depending only on $\rho_0$ and on $m.$
\medbreak
In what follows, we assume that $T$ and $R$ have been chosen so that \eqref{eq:smallv} is satisfied
by $v.$
%{}From \eqref{eq:J} and product estimates, we have  $I_1(v,w)\in 
%L^1(0,T;\dot B^{n/p-1}_{p,1})$ if also $w$ is in $E_p(T),$ and 
%\begin{equation}\label{eq:II1}
%\|I_1(v,w)\|_{L_T^1(\dot B^{n/p-1}_{p,1})}\lesssim\|Dv\|_{L_T^1(\dot B^{n/p}_{p,1})}
%\|\d_tw\|_{L_T^1(\dot B^{n/p-1}_{p,1})}.
%\end{equation}
Using the decomposition
$$\displaylines{\quad
(L_1-L_{\rho_0})u_L=(\rho_0^{-1}-1)\div\bigl(2\mu(\rho_0)D(u_L)
+\lambda(\rho_0)\div u_L\,\Id\bigr)\hfill\cr\hfill
+\div\bigl(2(\mu(\rho_0)-\mu(1))D(u)+(\lambda(\rho_0)-\lambda(1))\div u\,\Id\bigr),\quad}
$$
and composition inequalities \eqref{eq:compo1} and \eqref{eq:compo2}, we see that 
$(L_1-L_{\rho_0})u_L\in L^1(0,T;\dot B^{n/p-1}_{p,1})$ and 
\begin{equation}\label{eq:L1}
\|(L_1-L_{\rho_0})u_L\|_{L_T^1(\dot B^{n/p-1}_{p,1})}\lesssim 
\|a_0\|_{\dot B^{n/p}_{p,1}}(1+\|a_0\|_{\dot B^{n/p}_{p,1}})
\|Du_L\|_{L_T^1(\dot B^{n/p}_{p,1})}.
\end{equation}
Likewise, flow and composition estimates ensure  that
\begin{equation}\label{eq:II2}
\|I_i(v,w)\|_{L_T^1(\dot B^{n/p-1}_{p,1})}\lesssim
(1+\|a_0\|_{\dot B^{n/p}_{p,1}})\|Dv\|_{L_T^1(\dot B^{n/p}_{p,1})}
\|Dw\|_{L_T^1(\dot B^{n/p}_{p,1})}\quad\hbox{for $i=1,2,3$}
\end{equation}
and that
\begin{equation}\label{eq:II5}
\|I_4(v)\|_{L_T^1(\dot B^{n/p}_{p,1})}\lesssim T(1+\|a_0\|_{\dot B^{n/p}_{p,1}})(1+\|Dv\|_{L_T^1(\dot B^{n/p}_{p,1})}).
\end{equation}
So plugging the above inequalities in \eqref{eq:ut1} and keeping in mind
that  $v$  satisfies \eqref{eq:smallv}, we get after decomposing
$v$ into $\tilde v+u_L$: 
$$
\displaylines{\|\tilde u\|_{E_p(T)}\leq Ce^{C_{\rho_0,m}T}
(1+\|a_0\|_{\dot B^{n/p}_{p,1}})^2\Bigl((T+\|a_0\|_{\dot B^{n/p}_{p,1}}\|Du_L\|_{L_T^1(\dot B^{n/p}_{p,1})})
\hfill\cr\hfill+\|Du_L\|_{L_T^1(\dot B^{n/p}_{p,1})}^2
+\bigl(\|Du_L\|_{L_T^1(\dot B^{n/p}_{p,1})}+\|D\tilde v\|_{L_T^1(\dot B^{n/p}_{p,1})}\bigr)\|D\tilde v\|_{L_T^1(\dot B^{n/p}_{p,1})}\Bigr).}
$$
So, because $\tilde v\in B_{E_p(T)}(u_L,R),$
$$
\displaylines{
\|\tilde u\|_{E_p(T)}\leq Ce^{C_{\rho_0,m}T}(1+\|a_0\|_{\dot B^{n/p}_{p,1}})^2
\Bigl((T+\|a_0\|_{\dot B^{n/p}_{p,1}}\|Du_L\|_{L_T^1(\dot B^{n/p}_{p,1})})\hfill\cr\hfill
+(R+\|Du_L\|_{L_T^1(\dot B^{n/p}_{p,1})})\|Du_L\|_{L_T^1(\dot B^{n/p}_{p,1})}+R^2\Bigr).}
$$
Therefore, if we first choose $R$ so that for a small enough constant $\eta,$
\begin{equation}\label{eq:eta}(1+\|a_0\|_{\dot B^{n/p}_{p,1}})^2R\leq \eta
\end{equation}
 and then take $T$ so that 
\begin{equation}\label{eq:smalltime}
\begin{array}{c}
C_{\rho_0,m}T\leq\log2,\quad T\leq R^2,\quad
\|a_0\|_{\dot B^{n/p}_{p,1}}\|Du_L\|_{L_T^1(\dot B^{n/p}_{p,1})}\leq R^2,\\[2ex]
\|Du_L\|_{L_T^1(\dot B^{n/p}_{p,1})}\leq R,
\end{array}
\end{equation}
then we may conclude that $\Phi$ maps $\bar B_{E_p(T)}(u_L,R)$ into itself.

%%%%%%%%%%%%%%%%%%%%%%%%%%%%%%

\subsubsection*{Second  step: contraction estimates}
Let us now establish that, under Condition \eqref{eq:smalltime}, the map 
$\Phi$ is contractive. 
We consider two vector-fields $v^1$ and $v^2$ in $\bar B_{E_p(T)}(u_L,R),$ and 
set $u^1:=\Phi(v^1)$ and $u^2:=\Phi(v^2).$ Let $\du:=u^2-u^1$ and $\dv:=v^2-v^1.$
In order to prove that $\Phi$ is contractive, it is mainly a matter of applying Proposition \ref{p:lamerough} to 
$$\displaylines{
L_{\rho_0}\du=\rho_0^{-1}\div\Bigl((I_1(v^2,v^2)-I_1(v^1,v^1))\hfill\cr\hfill+(I_2(v^2,v^2)-I_2(v^1,v^1)) +(I_3(v^2,v^2)-I_3(v^1,v^1)) +(I_4(v^2)-I_4(v^1))
\Bigr)\cdotp}$$
So we have, given that $C_{\rho_0,m}T\leq\log2,$
 \begin{multline}\label{eq:ut2}
 \|\du\|_{E_p(T)}\leq C(1+\|a_0\|_{\dot B^{n/p}_{p,1}})
 \Bigl(\|I_1(v^2,v^2)-I_1(v^1,v^1)\|_{L_T^1(\dot B^{n/p}_{p,1})}\\+
 \|I_2(v^2,v^2)-I_2(v^1,v^1)\|_{L_T^1(\dot B^{n/p}_{p,1})}+\|I_3(v^2,v^2)-I_3(v^1,v^1)\|_{L_T^1(\dot B^{n/p}_{p,1})}
 +\|I_4(v^2)-I_4(v^1)\|_{L_T^1(\dot B^{n/p}_{p,1})}\Bigr)\cdotp
 \end{multline}
%The first term of the right-hand side may be bounded by means of \eqref{eq:II1}. As for the second term, 
%product estimates and \eqref{eq:dJ} imply that 
%$$
%\|(J_{v^1}-J_{v^2})\d_tv^2\|_{L_T^1(\dot B^{n/p-1}_{p,1})}\leq C\|D\dv\|_{L_T^1(\dot B^{n/p}_{p,1})}\|\d_tv^2\|_{L_T^1(\dot B^{n/p-1}_{p,1})}.
%$$
In order to deal with the first  term of the right-hand side, we use the decomposition
$$
\begin{array}{lll}I_1(v^2,v^2)-I_1(v^1,v^1)&=&\lambda(J_{v^2}^{-1}\rho_0)\bigl({}^T\!A_{v^2}:\nabla v^2\bigr)\bigl(\adj(DX_{v^2})-\adj(DX_{v^1})\bigr)\\
&+&\bigl(\adj(DX_{v^1})-\Id\bigr)\bigl(\lambda(J_{v^2}^{-1}\rho_0)-\lambda(J_{v^1}^{-1}\rho_0)\bigr)\bigl({}^T\!A_{v^2}:\nabla v^2\bigr)\\
&+&\bigl(\adj(DX_{v^1})-\Id\bigr)\lambda(J_{v^1}^{-1}\rho_0)\bigl(({}^T\!A_{v^2}-{}^T\!A_{v^1}):\nabla v^1+{}^T\!A_{v^2}:\nabla\dv\bigr)\\
&+&\hbox{terms pertaining to }\ \mu.\end{array}
$$
Taking advantage of product laws in Besov spaces, of composition estimates \eqref{eq:compo1} and \eqref{eq:compo2},
and of the flow estimates in the appendix, we deduce that for some constant $C_{\rho_0}$
depending only on~$\rho_0$:
$$
\|I_1(v^2,v^2)-I_1(v^1,v^1)\|_{L_T^1(\dot B^{n/p}_{p,1})}\leq C_{\rho_0}\|(Dv^1,Dv^2)\|_{L_T^1(\dot B^{n/p}_{p,1})}
\|D\dv\|_{L_T^1(\dot B^{n/p}_{p,1})}.
$$
Similar estimates may be proved for the next two terms 
 of the right-hand side of \eqref{eq:ut2}. Concerning the last one, we use the decomposition
 $$
 I_4(v^2)-I_4(v^1)=\bigl(\adj(DX_{v^1})-\adj(DX_{v^2})\bigr)P(J_{v^2}^{-1}\rho_0)-\adj(DX_{v^1})
 \bigl(P(J_{v^2}^{-1}\rho_0)-P(J_{v^1}^{-1}\rho_0)\bigr).
 $$
 Hence
 $$
 \|I_4(v^2)-I_4(v^1)\|_{L_T^1(\dot B^{n/p}_{p,1})}\leq C(1+\|a_0\|_{\dot B^{n/p}_{p,1}})T\|D\dv\|_{L_T^1(\dot B^{n/p}_{p,1})}.
 $$
 We end up with 
$$
\displaylines{
 \|\du\|_{E_p(T)}\leq C(1+\|a_0\|_{\dot B^{n/p}_{p,1}})^2\bigl(T+\|(Dv^1,Dv^2)\|_{L_T^1(\dot B^{n/p}_{p,1})}%+\|\d_tv^2\|_{L_T^1(\dot B^{n/p-1}_{p,1})}
 \bigr)\|D\dv\|_{L_T^1(\dot B^{n/p}_{p,1})}
 %\hfill\cr\hfill+\|Dv^1\|_{L_T^1(\dot B^{n/p}_{p,1})}\|\d_t\dv\|_{L_T^1(\dot B^{n/p-1}_{p,1})} .
}$$
Given that $v^1$ and $v^2$ are in $\bar B_{E_p(T)}(u_L,R),$
our hypotheses over $T$ and $R$ (with  smaller $\eta$ in \eqref{eq:eta} if need be) thus ensure that, say, 
$$
 \|\du\|_{E_p(T)}\leq \frac12 \|\dv\|_{E_p(T)}.
 $$
One can thus conclude that $\Phi$ admits a unique fixed point in $\bar B_{E_p(T)}(u_L,R).$

%%%%%%%%%%%%%%%%%%%%%%%%%%%%%%

\subsubsection*{Third step: Regularity of the density}

Granted with the above velocity field $u$ in $E_p(T),$ we set  $\rho:=J_u^{-1}\rho_0.$
By construction, the couple $(\rho,u)$ satisfies
\eqref{eq:lagrangian}.
Let us now prove that $a:=\rho-1$ is in  $\cC([0,T];\dot B^{n/p}_{p,1}).$ 
We have
$$
a=(J_u^{-1}-1)a_0+a_0.
$$
Given \eqref{eq:J} and using the fact that $Du\in L^1(0,T;\dot B^{n/p}_{p,1}),$  it is clear that $J_u^{-1}-1$
 belongs to $\cC([0,T];\dot B^{n/p}_{p,1}).$
Hence $a$  belongs to $\cC([0,T];\dot B^{n/p}_{p,1}),$ too.
Because $\dot B^{n/p}_{p,1}$ is continuously embedded in $L^\infty,$ Condition \eqref{eq:positive}
is fulfilled on $[0,T]$ (taking $T$ smaller if needed).

%%%%%%%%%%%%%%%%%%%%%%%

\subsubsection*{Last step: Uniqueness and continuity of the flow map}

We now consider two couples $(\rho_0^1,u_0^1)$ and $(\rho_0^2,u_0^2)$ of data fulfilling the assumptions
of Theorem \ref{th:main1} and we denote by $(\rho^1,u^1)$ and $(\rho^2,u^2)$ two solutions in $E_p(T)$
corresponding to those data. Setting $\du:=u^2-u^1,$ we see that
$$\displaylines{
L_{\rho_0^1}(\du)=(L_{\rho_0^1}-L_{\rho_0^2})(u^2)
+(\rho_0^1)^{-1}\div\Bigl(\sum_{j=1}^3\bigl((I_j^2(u^2,u^2)-I_j^2(u^1,u^1)\bigr)+(I_4^2(u^2)-I_4^2(u^1))\Bigr)\hfill\cr\hfill
+(\rho_0^1)^{-1}\div\Bigl(\sum_{j=1}^3((I_j^2-I_j^1)(u^1,u^1)+(I_4^2-I_4^1)(u^1)\Bigr),}$$
where $I_1^i,$ $I_2^i,$ $I_3^i$ and $I_4^i$ correspond to the quantities that have been defined just above \eqref{eq:Phi}, 
with density $\rho_0^i.$ Note that those terms may be bounded 
exactly as in the second step. So the only definitely new terms are  $(L_{\rho_0^1}-L_{\rho_0^2})(u^2)$ 
and the last line. As regards $(L_{\rho_0^1}-L_{\rho_0^2})(u^2),$  it may be decomposed into
$$\begin{array}{lll}
(L_{\rho_0^1}-L_{\rho_0^2})(u^2)&=&
\bigl((\rho_0^1)^{-1}-(\rho_0^2)^{-1}\bigr)\div \bigl(2\mu(\rho_0^1)D(u^2)+\lambda(\rho_0^1)\div u^2\Id\bigr)\\
&-&(\rho_0^2)^{-1}\div\bigl(2(\mu(\rho_0^2)-\mu(\rho_0^1))D(u^2)+(\lambda(\rho_0^2)-\lambda(\rho_0^1))\div u^2\Id\bigr).
\end{array}
$$
Hence, combining composition, flow and product estimates, we get for $t\leq T,$
$$
\|(L_{\rho_0^1}-L_{\rho_0^2})(u^2)\|_{L_t^1(\dot B^{n/p-1}_{p,1})}
\leq C_{\rho_0^1,\rho_0^2}\|\dr_0\|_{\dot B^{n/p}_{p,1}} \|Du^2\|_{L_t^1(\dot B^{n/p}_{p,1})}.
$$
It is not difficult to show that the other ``new" terms satisfy analogous estimates. 
Hence, applying Proposition \ref{p:lamerough} to the system that is satisfied by $\du,$ we discover that for $t\leq T,$
$$
\displaylines{
 \|\du\|_{E_p(t)}\leq C_{\rho_0^1,\rho_0^2}\bigl((t+\|(Du^1,Du^2)\|_{L_t^1(\dot B^{n/p}_{p,1})})
 \|D\du\|_{L_t^1(\dot B^{n/p}_{p,1})}\hfill\cr\hfill
 +\|\du_0\|_{\dot B^{n/p}_{p,1}}+\|\dr_0\|_{\dot B^{n/p}_{p,1}} (t+\|(Du^1,Du^2)\|_{L_t^1(\dot B^{n/p}_{p,1})})\bigr).}$$
Let us emphasize  that the constant $C_{\rho_0^1,\rho_0^2}$ depends only on $\rho_0^2$
through its norm, for  the integer $m$ used in Proposition \ref{p:lamerough}
corresponds to $\rho_0^1$ only. Hence if  $\dr_0$ is small enough then the above  inequality recasts in 
 $$
 \displaylines{
 \|\du\|_{E_p(t)}\leq C_{\rho_0^1}\bigl((t+\|Du^1\|_{L_t^1(\dot B^{n/p}_{p,1})}
 +\|\du\|_{E_p(t)})\|\du\|_{E_p(t)}
 \hfill\cr\hfill+\|\du_0\|_{\dot B^{n/p}_{p,1}}+\|\dr_0\|_{\dot B^{n/p}_{p,1}}(t+ \|Du^1\|_{L_t^1(\dot B^{n/p}_{p,1})})\bigr).}
$$
 An obvious bootstrap argument thus shows that if $t,$  $\du_0$ and $\dr_0$ are small enough then 
 $$
  \|\du\|_{E_p(t)}\leq 2C_{\rho_0}\bigl(\|\du_0\|_{\dot B^{n/p}_{p,1}}+\|\dr_0\|_{\dot B^{n/p}_{p,1}}\bigr).
  $$
    As regards the density, we have
$$
\da=J_{u^1}^{-1} \da_0+(J_{u^2}^{-1}-J_{u^1}^{-1})a_0^2.
 $$
 Hence for all $t\in[0,T],$ 
 $$
 \|\da(t)\|_{\dot B^{n/p}_{p,1}}
 \leq C(1+\|Du^1\|_{L_t^1(\dot B^{n/p}_{p,1})})\|\da_0\|_{\dot B^{n/p}_{p,1}}\|D\du\|_{L_t^1(\dot B^{n/p}_{p,1})}.
 $$
So we eventually get  uniqueness and continuity of the flow map on a small enough time interval. 
  Then iterating the proof yields uniqueness on the initial time interval $[0,T].$
   Note that it also yields Lipschitz continuity of the flow map for the velocity
  as for fixed  data $(\rho_0^1,u_0^1),$ one may find some neighborhood and common time interval on which 
  all the solutions constructed in the previous steps exist.

 %%%%%%%%%%%%%%%%%%%%%%%%%%%%%%%%%%%

  \subsection{Proof of Theorem \ref{th:main2}}
  
 For $u_0\in \dot B^{n/p-1}_{p,1}$ and $\rho_0\in (1+\dot B^{n/p}_{p,1}),$ the local existence for \eqref{eq:euler} may be 
proved directly (see \cite{CMZ1,D1})  \emph{but only under the assumption that $p\leq n$ in the case of 
nonconstant viscosity coefficients}.   Here we get the result (including uniqueness)
from  Theorem \ref{th:main1},  and under the sole assumption that $p<2n.$ 
This is a mere corollary of the following proposition which states the equivalence
of the systems \eqref{eq:euler} and \eqref{eq:lagrangian} in our functional  setting.
 \begin{prop}\label{p:equiv}
  Assume that the couple $(\rho,u)$ with $(\rho-1)\in\cC([0,T];\dot B^{n/p}_{p,1})$
and $u\in E_p(T)$ (with $1\leq p<2n$) is a solution to \eqref{eq:euler} such that
\begin{equation}\label{eq:smallu}
\int_0^T\|\nabla u\|_{\dot B^{n/p}_{p,1}}\,dt\leq c.
\end{equation}  
Let $X$ be the flow of $u$ defined in \eqref{lag}. 
Then the couple $(\bar\rho,\bar u):=(\rho\circ X,u\circ X)$ belongs
to the same functional space as $(\rho,u),$ and satisfies \eqref{eq:lagrangian}. 
\medbreak
Conversely, if  $(\bar\rho-1,\bar u)$ belongs to 
$\cC([0,T];\dot B^{n/p}_{p,1})\times E_p(T)$ and $(\bar\rho,\bar u)$ satisfies \eqref{eq:lagrangian} 
and, for a small enough constant $c,$
\begin{equation}\label{eq:smallbaru}
\int_0^T\|\nabla \bar u\|_{\dot B^{n/p}_{p,1}}\,dt\leq c
\end{equation} then 
the map
$X$ defined in \eqref{eq:lag} is a $ C^1$ (and in fact a locally $\dot B^{n/p+1}_{p,1}$) diffeomorphism over $\R^n$
and the couple $(\rho,u):=(\bar\rho\circ X^{-1},\bar u\circ X^{-1})$ satisfies
\eqref{eq:euler} and has the same regularity as $(\bar\rho,\bar u).$
\end{prop}
\begin{proof}
Let us first consider a solution $(\rho,u)$ to \eqref{eq:euler} with the above properties. 
Then, the definition of $X$   implies  that $DX-\Id$ is in $\cC([0,T];\dot B^{n/p}_{p,1}).$
In addition, Proposition \ref{p:change} ensures that   $(\bar\rho,\bar u):=(\rho\circ X,u\circ X)$
belongs to the same functional space as $(\rho,u),$ 
and  \eqref{eq:U1}, \eqref{eq:U2}, \eqref{eq:J} below imply that $A-\Id,$ $\adj(DX)-\Id$ and $J^{-1}-1$ are in $\cC([0,T];\dot B^{n/p}_{p,1}).$
Therefore the product laws for Besov spaces enable us to use  
the algebraic relations \eqref{eq:lag1}, \eqref{eq:lag2}, \eqref{eq:lag3} and \eqref{eq:lag4}
 whenever $p<2n.$ Therefore  $(\bar\rho,\bar u)$ fulfills \eqref{eq:lagrangian}. 
\smallbreak
Conversely, if we are given some solution $(\bar\rho,\bar u)$ in $\cC([0,T];(1+\dot B^{n/p}_{p,1}))\times E_p(T)$ 
to \eqref{eq:lagrangian} then one may check (see the appendix of \cite{D-Chambery}) that, under condition \eqref{eq:smallv},
 the ``flow'' $X(t,\cdot)$ of $\bar u$ defined by  
\begin{equation}\label{eq:defX}
X(t,y):=y+\int_0^t\bar v(\tau,y)\,d\tau
\end{equation}
is a $C^1$ diffeomorphism over $\R^n,$ and satisfies $DX-\Id\in\cC([0,T];\dot B^{n/p}_{p,1}).$
Hence one may construct
the Eulerian vector-field $u$ and Eulerian density by setting 
$$\rho(t,\cdot):=\rho\circ X^{-1}(t,\cdot)\quad\hbox{and}\quad u(t,\cdot):= u\circ X^{-1}(t,\cdot).$$
As above, the algebraic relations \eqref{eq:lag1}, \eqref{eq:lag2}, \eqref{eq:lag3} and \eqref{eq:lag4} hold
 whenever $p<2n.$ Hence  $(\rho,u)$ is a solution to \eqref{eq:euler}. 
 That $(\rho,u)$ has the desired regularity stems from Proposition \ref{p:change}.
\end{proof}
\medbreak\noindent{\em Proof of Theorem \ref{th:main2}.}  
We  consider  data $(\rho_0,u_0)$  with $\rho_0$ bounded away from $0,$ $(\rho_0-1)\in\dot B^{n/p}_{p,1}$ and $u_0\in\dot B^{n/p-1}_{p,1}$
 Then Theorem \ref{th:main1} provides  a local solution $(\bar \rho, \bar u)$
  to System \eqref{eq:lagrangian} in $\cC([0,T];(1+\dot B^{n/p}_{p,1}))\times E_p(T).$
  If $T$ is small enough then \eqref{eq:smallbaru} is satisfied 
so Proposition \ref{p:equiv} ensures that
$(\bar\rho\circ X^{-1},\bar u\circ X^{-1})$ is a solution of \eqref{eq:euler} in the desired functional 
space.
    \smallbreak
In order to prove uniqueness, we consider two solutions
$(\rho^1,u^1)$ and $(\rho^2,u^2)$ corresponding
to the same data $(\rho_0,u_0),$ and 
perform the Lagrangian  change of variable (pertaining to the flow of 
$u^1$ and $u^2$ respectively). 
The obtained vector-fields $\bar u^1$ and 
$\bar u^2$  are in $E_p(T)$ and both satisfy \eqref{eq:lagrangian}
\emph{with the same} $\rho_0$ and $u_0.$ Hence they coincide, as 
a consequence of the uniqueness part of  Theorem \ref{th:main1}. \qed

\appendix
\section{}
\setcounter{equation}{0}

\subsection{Change of coordinates}

Here we establish a result of regularity concerning changes of variables in Besov spaces.
Even though this  is somewhat classical (at least in nonhomogeneous Besov spaces), 
we did not find any reference in the literature of the estimates that we need. 
We here give a result in general Besov spaces $\dot B^s_{p,q},$ 
the definition of which may be found in e.g. \cite{BCD}.
\begin{prop}\label{p:change} Let  $X$ be a globally bi-Lipschitz diffeomorphism of $\R^n$
and $(s,p,q)$ with $1\leq p <\infty$ and $-n/p'<s <n/p$
(or just $-n/p'<s\leq n/p$ if $q=1$ and just $-n/p'\leq s<n/p$ if $q=\infty$).

Then   $a\mapsto a\circ X$ is a self-map  over $\dot B^s_{p,q}$ in the following cases:
\begin{enumerate}
\item  $s\in(0,1),$
\item $s\in(-1,0]$ and $J_{X^{-1}}$ is in the multiplier space
$\cM(\dot B^{-s}_{p',q'})$ defined in \eqref{eq:defmult},
\item $s\geq1$ and $(DX-\Id) \in \dot B^{n/p}_{p,1}.$
\end{enumerate}
\end{prop}
\begin{p}
Let us first assume that $s\in(0,1)$ and $q=p.$ Then one may use the classical characterization 
of the norm of $\dot B^s_{p,p}$ in terms of  finite differences (see e.g. \cite{BCD}) so as to write:
$$
 \|u\circ X\|_{\dot B^s_{p,p}(\R^n)}=
\biggl( \int_{\R^n}\!\!\int_{\R^n}\frac{|u(X(y))-u(X(x))|^p}{|y-x|^{n+sp}}\,dy\,dx\biggr)^{\frac 1p}\cdotp
 $$
Hence performing the change of variable $x'=X(x)$ and $y'=X(y),$ we get
$$
 \|u\circ X\|_{\dot B^s_{p,p}(\R^n)}=
\biggl( \int_{\R^n}\!\!\int_{\R^n}\frac{|u(y')-u(x')|^p}{|X^{-1}(y')-X^{-1}(x')|^{n+sp}}J_{X^{-1}}(y')
J_{X^{-1}}(x')\,dy'\,dx'\biggr)^{\frac 1p}
 $$
 whence 
 $$
 \|u\circ X\|_{\dot B^s_{p,1}(\R^n)}\leq \|J_{X^{-1}}\|_{L_\infty(\R^n)}^{\frac2p}
 \|DX\|_{L_\infty(\R^n)}^{s+\frac np}\|u\|_{\dot B^s_{p,1}(\R^n)}.
 $$
  The condition that $s<n/p$ ensures in addition that
 $u$ belongs to some Lebesgue space $L_{p^*}(\R^n)$ with $p^*<\infty$ (or in 
 the set of continuous functions going to $0$ at infinity if $q=1$ and $s=n/p$). Hence
 $u\circ X\in L_{p^*}(\R^n)$ too and one may thus conclude that
 $u\circ X\in \dot B^s_{p,p}(\R^n).$
 An interpolation argument then yields the desired result for any $s\in(0,1)$ and $q\in[1,+\infty].$
\medbreak
 The result for negative $s$ may be achieved by duality:  we have
  $$
 \|u\circ X\|_{\dot B^s_{p,q}(\R^n)}=\sup_{\|v\|_{\dot B^{-s}_{p',q'}(\R^n)}\leq1}\int_{\R^n} v(z) u(X(z))\,dz.
 $$
 Now, setting $x=X(z),$  we have
 $$
 \begin{array}{lll}
 \Int_{\R^n} v(z) u(X(z))\,dz&=&\Int_{\R^n} u(x)v(X^{-1}(x))\,dx,\\[1ex]
 &=&\Int_{\R^n} u(x)v(X^{-1}(x)) J_{X^{-1}}(x)\,dx,\\[2ex]
 &\leq& \|u\|_{\dot B^s_{p,q}(\R^n)}\|v\circ X^{-1} J_{X^{-1}}\|_{\dot B^{-s}_{p',q'}(\R^n)}.
 \end{array}
 $$
 So the definition of the multiplier space and the first part of the lemma allows to conclude. 
\medbreak
Finally, let us examine the cases of larger values of $s.$ 
If  $1<s<2$ then  one may write
$$
D(u\circ X)=(Du\circ X)\cdot DX.
$$
As $0<s-1<1,$ the first part of the proof ensures that $Du\circ X\in \dot B^{s-1}_{p,q}.$
As moreover $(DX-\Id)\in \dot B^{n/p}_{p,1},$ the standard product laws in Besov
spaces give the result.

If $2<s<3$ then we use the algebraic relation,
$$
D^2(u\circ X)= (D^2u\circ X)(DX,DX)+D^2X\cdot (Du\circ X).
$$
Hence the result follows from product laws and the previous result
applied with $s-1$ or $s-2.$

The higher values of $s$ may be achieved by induction, 
and the remaining cases ($s$ an integer) follow by interpolation.
The details are left to the reader.
\end{p}

%%%%%%%%%%%%%%%%%%%%%%%%%%%%%%%%%%%

\subsection{Some properties of Lagrangian coordinates}

Let us first derive  a few algebraic relations involving changes of  coordinates. 
We are  given a  $C^1$-diffeomorphism $X$ over $\R^n.$
For $H:\R^n\rightarrow\R^m,$ we agree that  $\bar H(y)= H(x)$ with $x=X(y).$ 
With this convention,  the chain rule writes
\begin{equation}\label{eq:chainrule}
D_y\bar H(y)=D_xH(X(y))\cdot D_yX(y)
\quad\hbox{with }\   (D_xH)_{ij}=\d_{x_j}H^i\ \hbox{ and }\ 
(D_yX)_{ij}=\d_{y_j}X^i,
\end{equation}
or, denoting $\nabla_y={}^T\!D_y$ and $\nabla_x={}^T\!D_x,$  
$$
\nabla_y\bar H(y)= (\nabla_yX(y))\cdot\nabla_xH(X(y)).
$$

Hence  we have 
\begin{equation}\label{eq:A}
D_xH(x)=D_y\bar H(y)\cdot A(y)\quad\hbox{with}\quad A(y)=(D_yX(y))^{-1}=D_xX^{-1}(x).
\end{equation}

\begin{lem}\label{l:div} Let  $K$ be   a $C^1$ scalar function over
$\R^n$ and $H,$  a $C^1$  vector-field.
  Let $X$ be a $C^1$ diffeomorphism such that  $J:=\det(D_yX)>0.$ Then the following relations hold true: 
\begin{eqnarray}
\label{eq:div2}
&&\overline{\nabla_xK}=J^{-1}\divy(\adj(D_yX)\bar K),\\[1ex]\label{eq:div1}
&&\overline{\divx  H}=J^{-1}\divy(\adj(D_yX)\bar H),\end{eqnarray}
where  $\adj(D_yX)$ stands for the adjugate of $D_yX.$\end{lem}
\begin{p}
The first item stems  from the following series of computations (based on integrations by parts, changes of variable 
and \eqref{eq:A})  which hold for any vector-field $\phi$ with coefficients in $\cC_c^\infty(\R^n)$:  
$$
\begin{array}{lll}
\Int \nabla_x  K(x)\cdot\phi(x)\,dx&=&-\Int K(x)\divx \phi(x)\,dx,\\[1ex]
&=&-\Int  \bar K(y)\overline{\divx \phi}(y) J(y)\,dy\\[1ex]
&=&-\Int   J(y)\bar K(y)\,D_y\bar\phi(y):A(y)\,dy,\\[1ex]
&=&\Int \bar \phi(y)\cdot \divy(\adj(D_yX)\bar K)(y)\,dy,\\[1ex]
&=&\Int\phi(x)\cdot \divy(\adj(D_yX)\bar K)(X^{-1}(x)) J^{-1}(X^{-1}(x))\,dx.\end{array}
$$
Proving  the second  item is similar.
\end{p}

Combining \eqref{eq:A},  \eqref{eq:div1} and  \eqref{eq:div2}, we deduce that  if $u:\R^n\rightarrow\R^n$  and $P:\R^n\rightarrow \R$ then
\begin{eqnarray}\label{eq:lag1}
&&\overline{\Delta_x u}=J^{-1}\divy(\adj(D_yX)\overline{\nabla_xu})
=J^{-1}\divy (\adj(D_yX){}^T\!A\nabla_y\bar u),\\\label{eq:lag2}
&&\overline{\nabla_x\divx u}=J^{-1}\divy(\adj(D_yX)\overline{\divx u})
=J^{-1}\divy (\adj(D_yX){}^T\!A:\nabla_y\bar u),\\\label{eq:lag3}
&&\overline{\nabla_xP}=J^{-1}\divy(\adj(D_yX)\bar P).
\end{eqnarray}
Note that we will use the above relations in the case where
$X$ is the flow of some time-dependent vector field $u,$ defined by the relation 
$$
X(t,y)=y+\int_0^t u(\tau,X(\tau,y))\,d\tau\quad\hbox{for all }\ t\in[0,T].
$$
Hence we will also have
\begin{equation}\label{eq:lag4}
J\,\overline{\d_t\rho+\div(\rho u)}=\d_t(J\bar\rho)\quad\hbox{and}\quad
 J\,\overline{\d_t(\rho u)+\div(\rho u\otimes u)}=\d_t(J\bar\rho \bar u).
\end{equation}

\smallbreak
Let us now establish  some estimates for  the flow $X_v$ of some given 
``Lagrangian'' vector field (that is  $X_v$ is defined by \eqref{eq:defX}).  
\begin{lem}\label{eq:flow} Let $p\in[1,+\infty)$ and $\bar v$ be in $E_p(T)$ 
satisfying  \eqref{eq:smallv}. Let $X_v$ be defined by \eqref{eq:defX}.  
Then we have for all $t\in[0,T],$
\begin{eqnarray}\label{eq:U1}
&&\|\Id-\adj(DX_v(t))\|_{\dot  B^{n/p}_{p,1}}\lesssim \|D\bar v\|_{L_t^1(\dot B^{n/p}_{p,1})},\\
\label{eq:U2}
&&\|\Id-A_v(t)\|_{\dot B^{n/p}_{p,1}}\lesssim \|D\bar v\|_{L_t^1(\dot B^{n/p}_{p,1})},\\
%\label{eq:U3}
%&&\|\d_t(\adj (DX))(t)\|_{\dot B^{n/p}_{p,1}}\lesssim  \|D\bar v(t)\|_{\dot B^{n/p}_{p,1}},\\\label{eq:U3b}
%&&\|\d_t(\adj (DX))(t)\|_{\dot B^{n/p-1}_{p,1}}\lesssim  \|D\bar v(t)\|_{\dot B^{n/p-1}_{p,1}}
%\quad\hbox{if }\ p<2n,\\
%\label{eq:U4}&&\|\adj(DX_v(t)){}^T\!A_v(t)-\Id\|_{\dot B^{n/p}_{p,1}}
%\lesssim \|D\bar v\|_{L_t^1(\dot B^{n/p}_{p,1})},\\
\label{eq:J}
&&\|J_v^{\pm1}(t)-1\|_{\dot B^{n/p}_{p,1}}\lesssim \|D\bar v\|_{L_t^1(\dot B^{n/p}_{p,1})}.
\end{eqnarray}
Furthermore, if $\bar w$ is a vector field such that $D\bar w\in L^1(0,T;\dot B^{n/p}_{p,1})$ then
\begin{eqnarray}\label{eq:U3}
\|(\adj(DX_v)D_{A_v}(\bar w)-D(\bar w))(t)\|_{\dot B^{n/p}_{p,1}}\lesssim
\|D\bar v\|_{L_t^1(\dot B^{n/p}_{p,1})}\|D\bar w\|_{L_t^1(\dot B^{n/p}_{p,1})},\\
\label{eq:U4}
\|(\adj(DX_v)\div_{\!A_v}(\bar w)-\div\bar w\:\Id)(t)\|_{\dot B^{n/p}_{p,1}}\lesssim
\|D\bar v\|_{L_t^1(\dot B^{n/p}_{p,1})}\|D\bar w\|_{L_t^1(\dot B^{n/p}_{p,1})}.
\end{eqnarray}
\end{lem}
\begin{p}
Recall that (see e.g. the appendix of \cite{DM-cpam})  for any $n\times n$ matrix $C$ we have
\begin{equation}\label{eq:adj}
\Id-\adj(\Id+C)=\bigl(C-({\rm Tr}\, C)\Id\bigr)+P_2(C),
\end{equation}
where the entries of the matrix $P_2(C)$ are at least quadratic polynomials. 
Applying this relation to the matrix $DX(t),$ and using the fact that
\begin{equation}\label{eq:DX}
DX_v(t,y)-\Id=\int_0^tD\bar v(\tau,y)\,d\tau,
\end{equation}
 we deduce that
$$
\Id-\adj(DX_v(t))=\int_0^t\bigl(D\bar v-\div\bar v\,\Id\bigr)\,d\tau
+P_2\biggl(\Bigl(\int_0^t D\bar v\,d\tau\Bigr)\biggr).
$$
Given that $\dot B^{n/p}_{p,1}$ is a Banach algebra and that \eqref{eq:smallv} holds,  we readily 
get \eqref{eq:U1}. 
\smallbreak
In order to prove \eqref{eq:U2}, we just use the fact that, under assumption \eqref{eq:smallv}, we have  
\begin{equation}\label{eq:C}
A_v(t)=(\Id+C_v(t))^{-1}=\sum_{k\in\N} (-1)^k(C_v(t))^k\quad\hbox{with}\quad
C_v(t)=\int_0^t D\bar v\,d\tau,
\end{equation}
and that $\dot B^{n/p}_{p,1}$ is a Banach algebra. 
\smallbreak
\smallbreak
As regards \eqref{eq:J}, we write
\begin{equation}\label{eq:Jv}
J_v(t,y)=1+\int_0^t\div v(\tau,X_v(\tau,y))\,J_v(\tau,y)\,d\tau
=1+\int_0^t (D\bar v:\adj(DX_v))(\tau,y)\,d\tau.
\end{equation}

Hence, if Condition  \eqref{eq:smallv} holds then we have \eqref{eq:J} for $J_v.$
In order to get the inequality for $J_v^{-1},$ it suffices to use the fact that
$$
J_v^{-1}(t,y)-1=(1+(J_v(t,y)-1))^{-1}-1=\sum_{k\geq1}(-1)^k\biggl(\int_0^tD\bar v:\adj(DX_v)\,d\tau\biggr)^k.
$$
%In order to prove the third inequality, we use the fact that, according to Lemma \ref{l:adj}, 
%we have
%$$
%\d_t\bigl(\adj(DX)\bigr)=\frac\d{\d t}\biggl(\Id+\int_0^t\bigl(\div\bar v\,\Id-D\bar v\bigr)\,d\tau
%+P_2\Bigl(\int_0^tD\bar v\,d\tau\Bigr)\biggr)\cdotp
%$$
%Hence
%$$
%\d_t\bigl(\adj(DX)\bigr)(t)=\bigl(\div\bar v(t)\Id-D\bar v(t)\bigr)+dP_2\Bigl(\int_0^tD\bar v\,d\tau\Bigr)\cdot D\bar v(t).
%$$ As the coefficients of $dP_2$ are polynomials of $n^2$ variables   that vanish 
%at $0,$ we get 
%$$
%\|\d_t(\adj (DX))(t)\|_{\dot B^{n/p}_{p,1}}\lesssim  
%\|D\bar v(t)\|_{\dot B^{n/p}_{p,1}}\biggl(1+\biggl\|\int_0^tD\bar v\,d\tau\biggr\|_{\dot B^{n/p}_{p,1}}\biggr),
%$$ hence \eqref{eq:U3}.
%Proving \eqref{eq:U3b} is similar: it is only a matter of using the continuity of
%the product from $\dot B^{n/p-1}_{p,1}\times\dot B^{n/p}_{p,1}$
%to  $\dot B^{n/p-1}_{p,1},$ if $p<2n.$
%\smallbreak

For proving \eqref{eq:U3}, we use the decomposition
$$\displaylines{\quad
2(\adj(DX_v)D_{A_v}(\bar w)-D(\bar w))=(\adj(DX_v)-\Id)(D\bar w+\nabla\bar w)
\hfill\cr\hfill+(\adj(DX_v)-\Id)\bigl(D\bar w\cdot(A-\Id)
+({}^T\!A-\Id)\cdot\nabla\bar w\bigr).}
$$
Hence the desired inequality stems from  \eqref{eq:U1} and \eqref{eq:U2}, and from 
the fact that $\dot B^{n/p}_{p,1}$ is a Banach algebra.
Inequality \eqref{eq:U4} is similar.
This completes the proof of the lemma.
\end{p}

\begin{lem} 
Let $\bar v_1$ and $\bar v_2$ be two vector-fields satisfying \eqref{eq:smallv},
and $\dv:=\bar v_2-\bar v_1.$ 
Then we have for all  $p\in[1,+\infty)$ and all $t\in[0,T]$ (with obvious notation):
\begin{equation}\label{eq:dA}
\|A_2(t)-A_1(t)\|_{\dot B^{n/p}_{p,1}} \lesssim 
 \|D\dv\|_{L_t^1(\dot B^{n/p}_{p,1})},
\end{equation}
\begin{equation}\label{eq:dAdj}
\|\adj(DX_2(t))-\adj(DX_1(t))\|_{\dot B^{n/p}_{p,1}} \lesssim 
 \|D\dv\|_{L_t^1(\dot B^{n/p}_{p,1})},
\end{equation}
\begin{equation}\label{eq:dJ}
\|J_2^{\pm1}(t)-J_1^{\pm1}(t)\|_{\dot B^{n/p}_{p,1}} \lesssim 
 \|D\dv\|_{L_t^1(\dot B^{n/p}_{p,1})}.
\end{equation}
%\begin{equation}\label{eq:dtdAdj}
%\|\d_t(\adj(DX_2)-\adj(DX_1))\|_{L_1(\R_+;\dot B^{n/p}_{p,1})} \lesssim 
 %\|D\dv\|_{L_1(\R_+;\dot B^{n/p}_{p,1})},
%\end{equation}
%\begin{equation}\label{eq:dtdAdjb}
%\|\d_t(\adj(DX_2)-\adj(DX_1))\|_{L_2(\R_+;\dot B^{n/p-1}_{p,1})} \lesssim 
 %\|D\dv\|_{L_2(\R_+;\dot B^{n/p-1}_{p,1})}\quad\hbox{if }\ p<2n.
%\end{equation}

 \end{lem}
 \begin{p}
 In order to prove the first inequality, we use the fact
 that, for $i=1,2,$ we have
 $$
 A_i=(\Id+C_i)^{-1}=\sum_{k\geq0}(-1)^kC_i^k\quad\hbox{with}\quad
 C_i(t)=\int_0^t D\bar v_i\,d\tau.
 $$
 Hence
$$
A_2-A_1= \sum_{k\geq 1} \Bigl(C_2^k-C_1^k\Bigr)
=\biggl(\int_0^tD\dv\,d\tau\biggr)\sum_{k\geq1}\sum_{j=0}^{k-1}
C_1^jC_2^{k-1-j}.
$$
So using the fact that $\dot B^{n/p}_{p,1}$ is a Banach algebra, it is easy to conclude
to \eqref{eq:dA}.
\smallbreak
The second inequality is a consequence of the decomposition \eqref{eq:adj} and of the Taylor formula which ensures
that, denoting $\dC:=C_2-C_1,$
$$
\adj(DX_2)-\adj(DX_1)=(\Tr(\dC))\Id-\dC+dP_2(C_1)(\dC)+\frac12d^2P_2(C_1,C_1)(\dC,\dC)+\dotsm
$$
where the coefficients of $P_2$ are polynomials of degree $n-1.$ 
As the sum is finite and $\dot B^{n/p}_{p,1}$ is a Banach algebra, we get \eqref{eq:dAdj}.

Proving the third inequality relies on similar arguments. It is only a matter of using \eqref{eq:Jv}. 
The details are left to the reader.

%\smallbreak
%In order to prove the last two estimates, it is only a matter of differentiating the above
%relation with respect to $t.$ Keeping in mind the definition of $C_1$ and $C_2,$ we get
%$$
%\d_t(\adj(DX_2)-\adj(DX_1))=(\div\dv)\Id-D\dv+dP_2(C_1)\cdot\d_t\dC+d^2P_2(C_1)(\d_tC_1,\dC)+\dotsm.
%$$
%Then using the product laws in Besov spaces yields the desired inequalities.
\end{p}

%%%%%%%%%%%%%%%%%%%%%%%%%%%%%%%%%%%%%%%%%%%

\subsection{Commutator and product estimates}

This last paragraph is devoted to the proof of commutator and product estimates that
have been used for investigating the Lam\'e system. Those proofs rely on the  following \emph{Bony decomposition}
(first introduced in \cite{Bony}) for the product of two functions:
\begin{equation}\label{eq:bony}
fg=T_fg+R(f,g)+T_gf.
\end{equation}
The paraproduct and remainder operators $T$ and $R$ are defined by 
$$
T_fg:=\sum_{j'\leq j-2} \dot\Delta_{j'}f\ddj g\quad\hbox{and}\quad
R(f,g):=\sum_{|j'-j|\leq1}  \dot\Delta_{j'}f\ddj g,
$$
where $(\ddj)_{j\in\Z}$ stands for some homogeneous Littlewood-Paley decomposition.
\begin{lem}\label{l:prod} Let $p$ be in $[1,+\infty]$ and the real numbers $\nu$ and $\sigma$ satisfy
$$
\nu\geq 0\quad\hbox{and}\quad -\min\biggl(\frac np,\frac n{p'}\biggr)<\sigma\leq\frac np-\nu.
$$
Then the following estimate holds true for all tempered distributions $f$ and $g$  over $\R^n$:
$$
\|fg\|_{\dot B^\sigma_{p,1}}\lesssim \|f\|_{\dot B^{n/p-\nu}_{p,1}}\|g\|_{\dot B^{\sigma+\nu}_{p,1}}.
$$
\end{lem}
\begin{proof}
The result relies on Bony decomposition \eqref{eq:bony}.
The standard continuity results for the paraproduct and remainder operators
ensure that (see e.g. \cite{BCD}, Chap. 2):
$$
\begin{array}{llll}
\|T_fg\|_{\dot B^\sigma_{p,1}}&\!\!\!\!\!\lesssim\!\!\!\!\!& \|f\|_{\dot B^{-\nu}_{\infty,1}}\|g\|_{\dot B^{\sigma+\nu}_{p,1}}
\quad&\hbox{if }\ \nu\geq0,\\[1ex]
\|T_gf\|_{\dot B^\sigma_{p,1}}&\!\!\!\!\!\lesssim\!\!\!\!\!& \|g\|_{\dot B^{\sigma+\nu-n/p}_{\infty,1}}\|f\|_{\dot B^{n/p-\nu}_{p,1}}
\quad&\hbox{if }\ \sigma+\nu-n/p\leq0,\\[1ex]
\|R(f,g)\|_{\dot B^\sigma_{p,1}}&\!\!\!\!\!\lesssim\!\!\!\!\!& \|f\|_{\dot B^{n/p-\nu}_{p,1}}\|g\|_{\dot B^{\sigma+\nu}_{p,1}}
\quad&\hbox{if }\ \sigma>-\min(n/p,n/p').
\end{array}
$$
So the result follows once noticed that 
$\dot B^s_{p,1}\hookrightarrow \dot B^{s-n/p}_{\infty,1}$ for any $s\in\R.$
\end{proof}

\begin{lem}\label{l:com}
Assume that $\sigma,$ $\nu$ and $p$ are such that
\begin{equation}\label{eq:condcom}
1\leq p\leq+\infty,\quad
0\leq\nu\leq \frac np\ \hbox{ and }\  -\min\biggl(\frac np,\frac n{p'}\biggr)-1<\sigma\leq\frac np-\nu.
\end{equation}
There exists a constant $C$ depending only on $\nu,$ $p,$  $\sigma$ and $n$ such that
for all $k\in\{1,\cdots,n\},$ we have for some sequence $(c_j)_{j\in\Z}$ with $\|c\|_{\ell^1(\Z)}=1$:
$$
\|\d_k[a,\ddj]w\|_{L^p}\leq 
Cc_j2^{-j\sigma}\|\nabla a\|_{B^{n/p-\nu}_{p,1}}\|w\|_{B^{\sigma+\nu}_{p,1}}\quad\hbox{for all }\ j\in\Z.
$$
\end{lem}
\begin{proof}
Taking advantage of the Bony decomposition \eqref{eq:bony}, we rewrite the commutator as\footnote{Here we use 
the notation $T'_uv:=T_uv+R(u,v).$}
\begin{equation}\label{eq:dec}
\d_k([a,\ddj]w)=
\underbrace{\d_k([T_{a},\ddj]w)}_{R_j^1}+\underbrace{\d_kT'_{\ddj w}a}_{R_j^2}
-\underbrace{\d_k\ddj T'_wa}_{R_j^3}.
\end{equation}
Arguing as in the proof of Lemma 6 in \cite{D5}, we get
$$
\|R_j^1\|_{L^p}\leq C\sum_{|j'-j|\leq4}\|\nabla\dot S_{j'-1}a\|_{L^\infty}\|\dot \Delta_{j'}w\|_{L^p}.
$$
Now, for $\nu\geq0,$ we have
$$
\|\nabla\dot S_{j'-1}a\|_{L^\infty}\leq C2^{j'\nu}\|\nabla a\|_{\dot B^{-\nu}_{\infty,1}}.
$$
Therefore,   for some sequence $(c_j)_{j\in\Z}$ in the unit sphere of $\ell^1(\Z),$
\begin{equation}\label{eq:com1a}
\|R_j^1\|_{L^p}\leq C c_j2^{-j\sigma}\|\nabla a\|_{\dot B^{-\nu}_{\infty,1}}\|w\|_{\dot B^{\sigma+\nu}_{p,1}}.
\end{equation}
To deal with $R_j^2,$ we use the fact that, owing to the  localization properties of the Littlewood-Paley decomposition, we have
$$R_j^2=\sum_{j'\geq j-2}
\d_k\bigl(S_{j'+2}\ddj w\,\dot\Delta_{j'}a\bigr).
$$ 
Hence, using the Bernstein and H\"older inequalities,
$$\begin{array}{lll}
\|R_j^2\|_{L^p}&\leq& C\Sum_{j'\geq j-2}\|S_{j'+2}\ddj w\|_{L^\infty}
\|\dot\Delta_{j'}\nabla a\|_{L^p},\\[1ex]
&\leq&C 2^{-j\sigma}\Sum_{j'\geq j-2}2^{(j-j')(\frac np-\nu)}\bigl(2^{j(\sigma+\nu-\frac np)}\|\ddj w\|_{L^\infty}\bigr)
\bigl(2^{j'(\frac np-\nu)}\|\dot\Delta_{j'}\nabla a\|_{L^p}\bigr).
\end{array}
$$ 
Therefore, by virtue  of convolution inequalities for series and because $n/p-\nu\geq0,$ 
\begin{equation}\label{eq:com2a}
\|R_j^2\|_{L^p}\leq C c_j2^{-j\sigma}\|\nabla a\|_{B^{n/p-\nu}_{p,1}}
\|w\|_{B^{\sigma+\nu-\frac np}_{\infty,1}}.
\end{equation}
Next,  from standard continuity results, we know that  the paraproduct and the remainder 
map $B^{\sigma+\nu}_{p,1}\times B^{n/p-\nu+1}_{p,1}$ in $B^{\sigma+1}_{p,1}$
whenever $\sigma+\nu-n/p\leq0$ and $\sigma+1>-\min(n/p,n/p').$ 
We thus have
\begin{equation}\label{eq:com4}
\|R_j^3\|_{L^p}
\leq 
Cc_j2^{-j\sigma}\|\nabla a\|_{B^{n/p-\nu}_{p,1}}\|w\|_{B^{\sigma+\nu}_{p,1}}.
\end{equation}
Putting Inequalities \eqref{eq:com1a}, \eqref{eq:com2a} and \eqref{eq:com4} together, and using classical embedding  completes the proof of the lemma.
\end{proof}
\begin{lem}\label{l:com2}
Let $A(D)$ be a Fourier multiplier of degree $0.$ Then the following estimate holds
$$
\|[A(D),q]w\|_{\dot B^{\sigma+1}_{p,1}}\leq C\|q\|_{\dot B^{1-\nu+n/p}_{p,1}}\|w\|_{\dot B^{\sigma+\nu}_{p,1}}
$$
whenever
$$
\nu\geq0\quad\hbox{and}\quad
-\min\biggl(\frac np,\frac n{p'}\biggr)-1<\sigma\leq \frac np-\nu.
$$
\end{lem}
\begin{proof} Taking advantage once again of Bony's decomposition, 
we decompose the commutator into
\begin{equation}\label{eq:decompo}
[A(D),q]w=[A(D),T_q]w+A(D)T'_wq-T'_{A(D)w}q.
\end{equation}
According to Lemma 2.99 in \cite{BCD}, we have for $\nu\geq0,$
$$
\|[A(D),T_q]w\|_{\dot B^{\sigma+1}_{p,1}}\leq C\|\nabla q\|_{\dot B^{-\nu}_{\infty,1}}\|w\|_{\dot B^{\sigma+\nu}_{p,1}}.
$$
Next, given that $A(D)$ is a homogeneous multiplier of degree $0,$ it maps any homogeneous Besov
space in itself. Therefore the last two terms of \eqref{eq:decompo} may be just bounded
according to standard continuity results for the paraproduct and remainder operators. 
\end{proof}

%{\footnotesize \noindent{\bf Acknowledgment.}}

\end{document}